\documentclass[12pt, reqno]{amsart}
\usepackage{amssymb,amsthm,amsfonts,amsmath, amscd}

\usepackage{hyperref}
\usepackage{mathrsfs}

\newcommand\Y{\mathbb Y}
\newcommand\Z{\mathbb Z}
\newcommand\C{\mathbb C}
\newcommand\R{\mathbb R}
\newcommand\T{\mathbb T}
\newcommand\GT{{\mathbb{S}}}

\newcommand\D{\mathbb D}
\newcommand\I{\mathbb I_M}
\renewcommand\AA{\mathbb A}

\newcommand\al{\alpha}
\newcommand\be{\beta}
\newcommand\ga{\gamma}

\newcommand\de{\delta}
\newcommand\De{\Delta}
\newcommand\ka{\varkappa}
\newcommand\La{\Lambda}
\newcommand\la{\lambda}
\newcommand\si{\sigma}
\newcommand\epsi{\varepsilon}
\newcommand\om{\omega}
\newcommand\Om{\Omega}

\newcommand\wt{\widetilde}
\newcommand\wh{\widehat}

\newcommand\const{\operatorname{const}}
\newcommand\Dim{\operatorname{Dim}}

\newcommand\Sym{\operatorname{Sym}}
\newcommand\sgn{\operatorname{sgn}}
\newcommand\Fun{{\operatorname{Fun}}}

\newcommand\dom{\operatorname{dom}}
\newcommand\supp{\operatorname{supp}}
\newcommand\Hom{\operatorname{Hom}}
\newcommand\Ind{\operatorname{Ind}}
\newcommand\Rep{\operatorname{Rep}}

\newcommand\RR{\mathscr R}
\newcommand\F{\mathscr F}
\newcommand\DD{\mathscr D}

\newcommand\A{\mathscr A}
\renewcommand\L{\mathscr L}
\newcommand\K{\mathscr K}
\newcommand\J{\mathscr J}

\newcommand\baM{{(a,b,M)}}
\newcommand\ba{{(a,b)}}
\newcommand\zw{{z,z',w,w'}}

\newcommand\one{\mathbf1}
\newcommand\One{\wt{\mathbf1}}

\newcommand\pd{\partial}

\newcommand\Minv{M_{\operatorname{inv}}}

\newtheorem{theorem}{Theorem}[section]
\newtheorem{proposition}[theorem] {Proposition}

\newtheorem{corollary}[theorem]{Corollary}
\newtheorem{lemma}[theorem]{Lemma}
\newtheorem{claim}[theorem]{Claim}

\theoremstyle{definition}
\newtheorem{definition}[theorem]{Definition}
\newtheorem{remark}[theorem]{Remark}

\numberwithin{equation}{section}

\begin{document}

\title[]
{The representation ring of the unitary groups and Markov processes of
algebraic origin}

\author{Grigori Olshanski}

\address{Institute for Information Transmission Problems, Moscow, Russia;
\newline\indent National Research University Higher School of Economics, Moscow, Russia}

\email{olsh2007@gmail.com}

\date{}

\dedicatory{To the memory of Andrei Zelevinsky}

\begin{abstract}
The paper consists of two parts. The first part introduces the representation
ring for the family of compact unitary groups $U(1)$, $U(2)$, \dots\,. This
novel object is a commutative graded algebra $R$ with infinite-dimensional
homogeneous components. It plays the role of the algebra of symmetric
functions, which serves as the representation ring for the family of finite
symmetric groups. The purpose of the first part is to elaborate on the basic
definitions and prepare the ground for the construction of the second part of
the paper.

The second part deals with a family of Markov processes on the dual object to
the infinite-dimensional unitary group $U(\infty)$. These processes were
defined in a joint work with Alexei Borodin (J. Funct. Anal. 2012). The main
result of the present paper consists in the derivation of an explicit
expression for their infinitesimal generators. It is shown that the generators
are implemented by certain second order partial differential operators with
countably many variables, initially defined as operators on $R$.

\end{abstract}

\maketitle

\tableofcontents

\section{Introduction}\label{sect1}

\subsection{Preliminaries: the symmetric group case}

The present paper deals with certain combinatorial and probabilistic aspects of
the representation theory of the infinite-dimensional unitary group
$U(\infty)$. A parallel theory also exists for the infinite symmetric group
$S(\infty)$. That theory is simpler and better developed, and it served as a
motivation for the present paper. So I start with a brief overview of some
relevant results which hold in the symmetric group case.

In the modern interpretation, classical Frobenius' construction \cite{Fr} of
irreducible characters of the symmetric groups $S(N)$ relies on the isomorphism
of graded algebras $\Rep(S(1),S(2),\dots)\simeq\Sym$, where $\Sym$ denotes the
algebra of symmetric functions and $\Rep(S(1),S(2),\dots)$ is our notation for
the \emph{representation ring of the family\/ $\{S(N): N=1,2,\dots\}$} of the
finite symmetric groups.

The algebra $\Rep(S(1),S(2),\dots)$ can be described as follows:
\begin{equation}\label{eq1.A}
\Rep(S(1),S(2),\dots):=\bigoplus_{N=0}^\infty \Rep^S_N
\end{equation}
where $\Rep^S_N$ is the space of class functions on $S(N)$, and the
multiplication
$$
\Rep^S_M\otimes \Rep^S_N\to \Rep^S_{M+N}
$$
is given by the operation of induction from $S(M)\times S(N)$ to $S(M+N)$.

(This definition should not be confused with that of the representation ring of
an individual group, see, e.g., Segal \cite{Segal}).

The algebra $\Rep(S(1),S(2),\dots)$ has a distinguished basis formed by the
irreducible characters of the symmetric groups. Under the isomorphism
$\Rep(S(1),S(2),\dots)\to\Sym$, called the \emph{characteristic map}, this
basis is transformed into the distinguished basis in $\Sym$ formed by the Schur
symmetric functions.

These facts are well known, see e.g. Macdonald \cite[Chapter I, Section
7]{Ma95}.

The \emph{infinite symmetric group} $S(\infty)$ is defined as the union of the
infinite chain
\begin{equation}\label{eq1.L}
S(1)\subset S(2)\subset \dots \subset S(N-1)\subset S(N)\subset \cdots
\end{equation}
of finite symmetric groups. For $S(\infty)$, the conventional notion of
irreducible characters is not applicable. However,  there exists a reasonable analog of \emph{normalized} irreducible characters (that is, irreducible characters divided by dimension). These are the so-called
\emph{extreme} characters whose definition, first suggested by Thoma
\cite{Tho64}, was inspired by the Murray--von Neumann theory of factors. Thoma
discovered that the extreme characters of $S(\infty)$ admit an explicit
description: they are parameterized by the points of the \emph{Thoma simplex}
$\Om^S$, a convex subset in the infinite-dimensional cube $[0,1]^\infty$. Note
that $\Om^S$ is compact in the product topology of $[0,1]^\infty$.

The \emph{dual object} to the group $S(N)$ is defined as the set $\wh{S(N)}$ of
its irreducible characters, and it can be identified with the set $\Y_N$ of
Young diagrams with $N$ boxes. Likewise, we regard the set of extreme
characters of the group $S(\infty)$ as (one of the possible versions of) the
dual object $\wh{S(\infty)}$ and identify it with the Thoma simplex $\Om^S$.

Vershik and Kerov \cite{VK81-Doklady}, \cite{VK81} initiated the
\emph{asymptotic theory of characters} (see also Vershik's foreword to
\cite{Kerov-book}). They explained how the extreme characters of the group
$S(\infty)$ arise from the normalized irreducible characters of the groups
$S(N)$ in a limit transition as $N$ goes to infinity. In the asymptotic theory
of characters, the algebra $\Sym$ still plays an important role. In particular,
the so-called \emph{ring theorem} of Vershik and Kerov says that the extreme
characters of $S(\infty)$ are in a one-to-one correspondence with those linear
functionals on $\Sym$ that are multiplicative, take nonnegative values on the
basis of Schur functions, and vanish on the principal ideal
$(e_1-1)\subset\Sym$, where $e_1$ is the first elementary symmetric function
(see Vershik-Kerov \cite{VK-RingTheorem} and also Gnedin-Olshanski
\cite{GO-zigzag}).

Now I proceed to probabilistic results. First, note that the embedding
$S(N-1)\subset S(N)$ gives rise, by duality, to a canonical ``link''
$\wh{S(N)}\dasharrow\wh{S(N-1)}$. Here by a link $X\dasharrow Y$ between two
spaces I mean a ``generalized map'' which assigns to every point of $X$ a
probability distribution on $Y$; in other words, a link is given by a Markov
kernel (which in our case is simply a stochastic matrix). As explained in
Borodin-Olshanski \cite{BO-MMJ}, the dual object $\wh{S(\infty)}$ can be viewed
as the projective limit of the chain
\begin{equation}\label{eq1.J}
\wh{S(1)}\dashleftarrow\wh{S(2)}\dashleftarrow\dots\dashleftarrow\wh{S(N-1)}
\dashleftarrow\wh{S(N)} \dashleftarrow\cdots
\end{equation}
taken in an appropriate category with morphisms given by Markov kernels. Thus,
$S(\infty)$ is an inductive limit group while its dual object $\wh{S(\infty)}$
is obtained by taking a kind of projective limit.

In \cite{BO-PTRF09}, Borodin and I constructed a two-parameter family of
continuous time Markov processes on the Thoma simplex. Our work was inspired by
our previous study of the problem of harmonic analysis on $S(\infty)$ and
substantially used the canonical links from \eqref{eq1.J}. We proved that the
Markov processes in question have continuous sample trajectories and
consequently are diffusion processes. The proof relied on the computation of
the infinitesimal generators of the processes: we showed that the generators
are given by certain second order differential operators initially acting on
the the quotient algebra $\Sym/(e_1-1)$. To relate the?se operators to Markov
processes we used the fact that there is a canonical embedding
\begin{equation}\label{eq1.K}
\Sym/(e_1-1) \hookrightarrow C(\Om^S),
\end{equation}
where $C(\Om^S)$ denotes the Banach algebra of continuous functions on the
compact space $\Om^S$.

\subsection{The results}

Let us turn to the compact unitary groups. They are organized into a chain
similar to \eqref{eq1.L},
\begin{equation*}
U(1)\subset U(2)\subset \dots \subset U(N-1)\subset U(N)\subset \cdots,
\end{equation*}
and we set $U(\infty):=\bigcup_{N=1}^\infty U(N)$. The extreme characters of
the group $U(\infty)$ were first investigated by Voiculescu \cite{Vo76}. They
are parameterized  by the points of an infinite-dimensional space $\Om$, which
can be realized as a convex subset in the product of countably many copies of
$\R_+$ (see Subsection \ref{sect3.A} below). Note that $\Om$ is locally
compact. Like the dual object to $S(\infty)$, the space $\Om=\wh{U(\infty)}$
can be identified with the projective limit of the dual chain
\begin{equation}\label{eq1.N}
\wh{U(1)}\dashleftarrow\wh{U(2)}\dashleftarrow\dots\dashleftarrow\wh{U(N-1)}
\dashleftarrow\wh{U(N)} \dashleftarrow\cdots
\end{equation}

Although the groups $S(\infty)$ and $U(\infty)$ are structurally very
different, there is a surprising similarity in the description of their
characters. An explanation of this phenomenon is suggested in Borodin-Olshanski
\cite{BO-MMJ}.

Here is a brief description of what is done in the present paper.

1. The attempt to extend the definition of the representation ring to the
family of the unitary groups leads us to a novel object --- a certain graded
algebra $R$, which plays the role of the algebra $\Sym$.

2. An analog of the embedding \eqref{eq1.K} is found. As explained below, it
may be viewed as a kind of Fourier transform on $U(\infty)$.

3. The main result is the computation of the infinitesimal generators for the
four-parameter family of Markov processes on $\Om$, previously constructed in
Borodin-Olshanski \cite{BO-GT-Dyn}. It is shown that the generators in question
are implemented by certain second order partial differential operators,
initially defined on $R$.

Now I will describe the results in more detail. As will be clear, for all the
similarities between $S(\infty)$ and $U(\infty)$, the unitary group case turns
out to be substantially more complicated.

\subsection{The representation ring for the unitary groups: the algebra $R$}

At first it was unclear to me if there is a good analog of the representation
ring  for the family $\{U(N)\}$. The difficulty here is that, in contrast to
the case of finite symmetric groups, induced characters have infinitely many
irreducible constituents. Therefore, directly following the definition of
$\Rep(S(1),S(2),\dots)$ we  see that products of basis elements are infinite
sums; how to deal with them? The proposed solution is to enlarge the space and
allow infinite sums. This leads to the following definition:

The algebra $R$, the suggested analog of the algebra $\Sym$, is the graded
algebra of formal power series of bounded degree, in countably many variables
each of which has degree 1. The variables are denoted by $\varphi_n$, where $n$
ranges over $\Z$.

Recall that $\Sym$ is the projective limit of polynomial algebras:
\begin{equation}\label{eq1.H}
\Sym=\varprojlim \C[e_1,\dots,e_k],
\end{equation}
where $k\to\infty$ and $e_1,e_2,\dots$ are the elementary symmetric functions.

Likewise, $R$ also can be represented as the projective limit of polynomial
algebras:
\begin{equation}\label{eq1.I}
R =\varprojlim \C[\varphi_{-l}, \dots,\varphi_{k}],
\end{equation}
where $k, l\to+\infty$.

A substantial difference is that $\deg e_k=k$, while $\deg\varphi_n=1$ for all
$n\in\Z$. Because of this, the homogeneous components of $\Sym$ have finite
dimension, while those of $R$ are infinite-dimensional. Nevertheless, it turns
out that the projective limit realization \eqref{eq1.I} is a kind of finiteness
property which can be efficiently exploited.

As in the case of the algebra $\Sym$, in $R$ there exist various interesting
bases, but these are \emph{topological bases}. Two bases are of particular
importance for the purpose of this paper. They are denoted as $\{\varphi_\la\}$
and $\{\si_\la\}$, where the subscript $\la$ ranges over the set of highest
weights of all unitary groups. The basis $\{\varphi_\la\}$ is formed by the
monomials in letters $\varphi_n$ and is similar to the multiplicative basis in
$\Sym$ generated by the elementary symmetric functions. The basis $\{\si_\la\}$
is an analog of the Schur functions. The interplay between these two bases
plays an important role in the derivation of the main result.

By the Schur-Weyl duality, the representation ring for the \emph{family}
$\{S(N)\}$ is isomorphic to a certain representation ring of a \emph{single}
object --- the Lie algebra $\mathfrak{gl}(\infty)$. Likewise, using the fermion
version of the Howe duality one can identify the representation ring for the
family $\{U(N)\}$ with a certain representation ring for the Lie algebra
$\mathfrak{gl}(2\infty)$ (for more detail, see Subsection \ref{sect2.B} below).

\subsection{What is the Fourier transform on $U(\infty)$?}

Let us consider first a finite group $G$ and let $\Minv(G)$ denote the space of
complex measures on $G$, invariant with respect to inner automorphisms.
Next, let $\wh G$ stand for the set of normalized irreducible characters and
$\Fun(\wh G)$ denote the space of functions on  $\wh G$. By integrating a
character $\chi\in\wh G$ against a measure $m\in\Minv(G)$ we get a linear map
$$
F:  \Minv(G)\to \Fun(\wh G).
$$
Using the functional equation for normalized irreducible characters one sees
that $F$ turns the convolution product of measures into the pointwise product
of functions. So $F$ is a reasonable version of Fourier transform.

More generally, the above definition of Fourier transform $F$ works perfectly
when $G$ is a compact group. Then as $\Minv(G)$ one can still take the space of
invariant complex measures on $G$ or, if $G$ is a Lie group, the larger space
of invariant distributions or else an appropriate subspace therein, depending
on the situation.

But what happens for $G=S(\infty)$ or $G=U(\infty)$? The dual object $\wh G$
has been defined, and one knows that it is large enough in the sense that the extreme characters of these groups separate the
conjugacy classes. The problem is that the above definition of
$\Minv(G)$ no longer works. For instance, the only invariant \emph{finite}
measure on $S(\infty)$ is the delta measure at the unit element.

This difficulty can be resolved as follows. For a group $G$ which is an
inductive limit of compact groups $G(N)$ we define
$$
\Minv(G):=\varinjlim \Minv(G(N)),
$$
where the map $\Minv(G(N-1))\to\Minv(G(N))$ is given by averaging over the
action of the group of inner automorphisms of $G(N)$. In more detail, given a measure $M\in\Minv(G(N-1))$, its image in $\Minv(G(N))$ is defined as
$$
\int_{g\in G(N)}M^g dg,
$$
where $M^g$ denotes the transformation of $M$ (which we transfer from $G(N-1)$ to $G(N)$) under the conjugation by an element $g\in G(N)$, and $dg$ denotes the normalized Haar measure on $G(N)$. 

In the case of $G=S(\infty)$ it is readily verified that $\Minv(S(\infty))$ can
be identified, in a natural way, with the quotient algebra $\Sym/(e_1-1)$, and
then the Fourier transform just defined coincides with the map \eqref{eq1.K}.

In the case $G=U(\infty)$ the situation is more delicate. In the first
approximation, the analog of $\Sym/(e_1-1)$ is the quotient algebra $R/J$,
where $J$ is the following principal ideal
\begin{equation}\label{eq1.M}
J:=(\varphi-1), \qquad \varphi:=\sum_{n\in\Z}\varphi_n.
\end{equation}
However, this algebra is too large and one has to narrow it in order for the
Fourier transform to be well defined. We discuss two variants of doing this,
both of which seem to be quite natural. Note that there are also many other
possibilities: they depend on the concrete choice of the spaces $\Minv(U(N))$.
I did not go too far in this direction, because for the main result it was
sufficient to dispose of the simplest way to relate the algebra $R$ to the
space $\wh{U(\infty)}=\Om$.

Note that in a number of cases involving those of $G=S(\infty)$ and $G=U(\infty)$, the set of conjugacy classes of $G$ can be endowed with a natural semigroup structure (see \cite{Ols-AA}, \cite{Ols-Semigroups}, \cite{Ols-GordonBreach}). Then one may endow $\Minv(G)$ with a multiplication, which is an analog of convolution product and which turns into pointwise multiplication on $\wh G$  under a suitable version of  Fourier transform.

\subsection{The Markov generators}

The Markov processes on $\Om$ constructed  in Borodin-Olshanski
\cite{BO-GT-Dyn} depend on four complex parameters $\zw$ subject to certain
constraints (see Definition \ref{def6.A}). Let us ignore  for a moment the
constraints, so that $\zw$ are arbitrary complex numbers, and consider a formal
second order partial differential operator
\begin{multline}\label{eq1.C}
\D_{\zw}=\sum_{n_1,n_2\in\Z} A_{n_1n_2}(\dots,\varphi_{-1},
\varphi_0,\varphi_1,\dots)\frac{\pd^2}{\pd\varphi_{n_1}\pd\varphi_{n_2}}\\
+\sum_{n\in\Z} B_n(\dots,\varphi_{-1}, \varphi_0,\varphi_1,\dots;
\zw)\frac{\pd}{\pd\varphi_n},
\end{multline}
where the variables $\varphi_n$ are indexed by integers $n\in\Z$, the second
order coefficients $A_{n_1n_2}$ are certain (complicated) quadratic expressions
in the variables, and the first order coefficients $B_n$ are certain linear
expressions which involve the parameters, see the explicit formulas
\eqref{eq4.B} and \eqref{eq4.C} below.

The main result of the paper can be informally stated as follows.

\begin{theorem}\label{thm1.A}
Assume that the quadruple $(\zw)$ satisfies the necessary constraints, so that
the construction of \cite{BO-GT-Dyn} provides a Markov process $X_\zw$ on
$\Om$. Then the generator of $X_\zw$ is implemented by the differential
operator\/ $\D_\zw$.
\end{theorem}

A rigorous version is given in Theorem \ref{thm7.B}.

Note that the Markov generator in question is defined on a dense subspace of
$C_0(\Om)$, the Banach space of continuous functions on $\Om$ vanishing at
infinity. To relate  such an operator with an operator acting on $R$ we use the
Fourier transform discussed in the preceding subsection. Here we use the fact
that $\D_\zw$  preserves the principal ideal $J\subset R$ (see \eqref{eq1.M}
above) and so also acts on $R/J$.

The operator $\D_\zw$ is well adapted to the basis $\{\varphi_\la\}$ in $R$
while the Markov generators are initially defined by their action on another
basis, $\{\si_\la\}$. This is the main source of difficulty in the proof of the
main theorem: transition from one basis to another one is achieved by rather
long computations.

The construction of the processes $X_\zw$ in our work \cite{BO-GT-Dyn} is based
on a limit transition along the chain \eqref{eq1.N}: we find jump processes on
the dual objects $\wh{U(N)}$ which are consistent with the ``links''
$\wh{U(N)}\dasharrow\wh{U(N-1)}$. The key idea is very simple but the
construction is formal and it drastically differs from the approaches used by
probabilists. So the intriguing problem is to understand what is the nature of the
processes $X_\zw$ and what can be explicitly computed. The computation of the
Markov generators in the present paper is the first step in this direction.

The fact that the Markov generators are implemented by differential operators
makes plausible the conjecture that the sample trajectories of the processes
are continuous (the diffusion property). In the symmetric group case (see
Borodin-Olshanski \cite{BO-PTRF09}) we give a simple proof of the diffusion
property for the processes on the Thoma simplex $\Om^S$ using the realization
of their generators as differential operators on $\Sym$. However, the structure
of the differential operator $\D_\zw$ is substantially more complicated,
because, in contrast to the symmetric group case, the coefficients $A_{n_1n_2}$
are given by infinite series. This is an obstacle to extending the approach of
\cite{BO-PTRF09}.

It seems that the Markov generators cannot be written in terms of the natural
coordinates on $\Om$, and the same holds in the models related to $S(\infty)$,
studied in \cite{BO-PTRF09} and \cite{BO-EJP} (a possible explanation is that
the coordinate functions do not enter the domain of the generators, see in this
connection the discussion in Petrov \cite[Remark 5.4]{Petrov-FAA} concerning a
simpler model). This is why one needs to use a more involved construction using
the algebra $R$ (or, in the symmetric group case, the algebra $\Sym$).

\subsection{Lifting of multivariate Jacobi differential operators to algebra $R$}

Let $m=1,2,3,\dots$\,. The \emph{Jacobi partial differential operator} in $m$
variables $t_1,\dots,t_m$ is given by
\begin{equation}\label{eq1.P}
D^{(a,b)}_m := \sum_{i=1}^m \left(t_i(1-t_i)\frac{\partial^2}{\partial
t_i^2}+\left[b+1-(a+b+2)t_i+\sum_{j:\, j\ne
i}\frac{2t_i(1-t_i)}{t_i-t_j}\right]\frac{\partial}{\partial t_i}\right).
\end{equation}
Here $a$ and $b$ are parameters. In the simplest case $m=1$ this operator turns into the familiar hypergeometric
ordinary differential operator
$$
D^{(a,b)}=t(1-t)\frac{d^2}{dt^2}+[b+1-(a+b+2)t]\frac{d}{dt}.
$$
The operator $D^{(a,b)}$ is attached to the Jacobi orthogonal polynomials with
the weight function $t^b(1-t)^a$ on the unit interval $0\le t\le 1$, that is,
the Jacobi polynomials are just the polynomial eigenfunctions of $D^{(a,b)}$.

In the case of several variables, despite the singularities on the hyperplanes
$t_i=t_j$, the operator $D^{(a,b)}_m$ is well defined on the space of symmetric
polynomials in $t_1,\dots,t_m$ and is diagonalized in the basis of $m$-variate
symmetric Jacobi polynomials. The latter polynomials are a particular case of
the Heckman-Opdam orthogonal polynomials, which corresponds to the root system
$BC_m$ and a special choice of the ``Jack parameter'' (see e.g. Heckman
\cite{Heckman}, Koornwinder \cite{Koornwinder}). The operator $D^{(a,b)}_m$ is
well known; it appeared (in a more general form involving the Jack parameter)
in many works, see, e.g., Baker-Forrester \cite{BF}.

Given $m$, let us fix two nonnegative integers $k$ and $l$ such that $k+l=m$.
We assume that $m+1$ variables $\varphi_{-l},\dots,\varphi_k$ are expressed through
$m$ variables $t_1,\dots,t_m$ via
$$
\sum_{n=-l}^k\varphi_nu^n=\prod_{i=1}^k
(t_i+(1-t_i)u)\cdot\prod_{i=k+1}^m(1-t_i+t_iu^{-1}),
$$
where the left-hand side should be viewed as a generating series for $\varphi_{-l},\dots,\varphi_k$  with an auxiliary indeterminate $u$ (then, by equating the coefficients of monomials $u^n$ in the both sides, we can write $\varphi_n$'s as polynomials in $t_i$'s). 
Setting $u=1$ one sees that the constraint $\sum_{n=-l}^k\varphi_n=1$ holds. Moreover, we
may identify the algebra $\Sym_m$ of symmetric polynomials in variables
$t_1,\dots,t_m$ with
$$
\wh
R(k,-l):=\C[\varphi_{-l},\dots,\varphi_k]\big/\left(\sum_{n=-l}^k\varphi_n-1\right),
$$
the quotient by the principal ideal generated by the element
$\sum_{n=-l}^k\varphi_n-1$.

In the next theorem  we regard the same algebra $\wh R(k,-l)$ as the quotient
$R/J(k,-l)$, where $J(k,-l)$ denotes the ideal of $R$ generated by the elements
\begin{equation}\label{eq1.O}
\varphi_{k+1}, \varphi_{k+2}, \dots; \quad \varphi_{-l}, \varphi_{-l-1},
\dots;\quad \varphi_{-l}+\dots+\varphi_k-1.
\end{equation}
Note that the ideal does not change if  $\varphi_{-l}+\dots+\varphi_k-1$ is
replaced by $\varphi-1$, where $\varphi$ is defined in \eqref{eq1.M} above.

{}From the proof of Theorem \ref{thm1.A} one can extract the following fact:

\begin{theorem}\label{thm1.B}
Let us assume that parameters $z$ and $w$ are nonnegative integers, which are
not both $0$. Let us denote them by $k$ and $l$, respectively.

In this special case the differential operator $\D_\zw$ preserves the ideal
$J(k,l)\subset R$ and so determines an operator on $R/J(k,-l)=\wh R(k,-l)$. The
latter  operator coincides with the $(k+l)$-variate Jacobi operator
\eqref{eq1.P} with parameters $a=z'-k$, $b=w'-l$.
\end{theorem}

This fact clarifies the nature of the differential operator $\D_\zw$. Indeed,
from Theorem \ref{thm1.B} one can see that the sophisticated expression for
$\D_\zw$ appears as the result of formal analytic extrapolation, with respect
to parameters $(k,l,a,b)$, of the Jacobi differential operators
$D^{(a,b)}_{k+l}$ rewritten in a new set of variables. Note that as $k$ and $l$
increase, the ideals $J(k,-l)$ decrease and their intersection
$\cap_{k,l=1}^\infty J(k,-l)$ coincides with the principal ideal $J\subset R$
generated by the sole element $\varphi-1$. Note also that the extrapolation
procedure is purely formal, because the integers $k$ and $l$, whose sum $m=k+l$
initially represents the number of variables, finally turn into complex
parameters.

It is interesting to compare this picture with what is done in the work of
Sergeev and Veselov \cite{SV} which deals with the same Jacobi differential
operators (involving the additional ``Jack parameter''). However, in \cite{SV}
the operators are lifted to the algebra $\Sym$, while our target space is the
algebra $R$. The initial motivation of Sergeev and Veselov is also different:
they used the lifting to $\Sym$ as a tool for constructing super versions of
quantum integrable systems in finite dimensions, while our interest is in
infinite-dimensional Markov dynamics. (See also the papers Desrosiers-Halln\"as
\cite{DH}, Olshanski \cite{Ols-LaguerreNote}, \cite{Ols-Laguerre} --- in all
these works the target space is $\Sym$.)

\subsection{Organization of the paper}

Section \ref{sect2} introduces the algebra $R$  and Section \ref{sect3} relates
it to the dual object $\wh{U(\infty)}$. Section \ref{sect4} introduces the
differential operator $\D_\zw$. In Sections \ref{sect5} and \ref{sect6} we
recall some general facts about Feller Markov processes, next describe the
``method of intertwiners'' \cite{BO-GT-Dyn}, and then explain how it produces a
special family of Markov processes on $\wh{U(\infty)}$ out of continuous time
Markov chains on the discrete sets $\wh{U(N)}$. In Section \ref{sect7} we
formulate the main theorem and outline the plan of its proof.  The proof itself
occupies Sections \ref{sect8} and \ref{sect9}. The last Section \ref{sect10} is
an appendix, where we prove the uniform boundedness of multiplicities in
certain induced representations of compact groups; this fact was used in
Section \ref{sect3}.

\subsection{Acknowledgement}
I am grateful to Igor Frenkel for an important comment which I used in
Subsection \ref{sect2.B}, and to Vladimir L. Popov who confirmed that the
statement of Proposition \ref{prop10.A} is true and communicated its proof to
me. I am also grateful to the anonymous referee for valuable suggestions. This research was partially supported by a grant from Simons Foundation
(Simons-IUM Fellowship) and by
the RFBR grant 13-01-12449.

\section{The algebra $R$}\label{sect2}

\subsection{Definition of algebra $R$}\label{sect2.A}

Throughout the paper $\{\varphi_n\}$ stands for a doubly infinite collection of
formal variables indexed by arbitrary integers $n\in\Z$.

We define $R$ as the commutative complex unital algebra formed by arbitrary
formal power series of bounded degree, in variables $\varphi_n$, $n\in\Z$. Here
we assume that $\deg\varphi_n=1$ for every $n$. The algebra $R$ is graded: we
write $R=\bigoplus_{N=0}^\infty R_N$, where the elements of the $N$th
homogeneous component $R_N$ have the form
\begin{equation}\label{eq2.A}
\psi=\sum_{n_1\ge \dots\ge
n_N}a_{n_1,\dots,n_N}\varphi_{n_1}\dots\varphi_{n_N}
\end{equation}
with no restriction on the complex coefficients $a_{n_1,\dots,n_N}$.

Equivalently, $R$ can be defined as a projective limit of polynomial algebras.
Namely, for a pair of integers $n_+\ge n_-$ we set
\begin{equation*}
R(n_+,n_-):=\C[\varphi_{n_-},\varphi_{n_-+1},\dots,\varphi_{n_+-1},\varphi_{n_+}].
\end{equation*}
Then one can write
$$
R=\varprojlim R(n_+,n_-), \qquad n_+\to+\infty, \quad n_-\to-\infty,
$$
where the limit is taken in the category of graded algebras.

We call the natural homomorphisms $R\to R(n_+,n_-)$ the \emph{truncation maps}.
Let $I(n_+,n_-)$ denote the kernel of the truncation $R\to R(n_+,n_-)$. As
$n_\pm\to\pm\infty$, the ideals $I(n_+,n_-)$ decrease and their intersection
equals $\{0\}$. We take these ideals as the base of a topology in $R$, which we
call the \emph{$I$-adic topology}.

Following Weyl \cite{Weyl} we define a \emph{signature of length} $N$ as an
arbitrary vector $\la=(\la_1,\dots,\la_N)\in\Z^N$ with weakly decreasing
coordinates: $\la\ge\dots\ge\la_N$. The set of all such vectors is denoted by
$\GT_N$. In particular, $\GT_1=\Z$. By agreement, $\GT_0$ consists of a single
element denoted by $\varnothing$.

With a signature $\la\in\GT_N$ we associate a monomial of degree $N$,
$$
\varphi_\la:=\varphi_{\la_1}\dots\varphi_{\la_N},
$$
and we agree that $\varphi_\varnothing=1$. With this notation, \eqref{eq2.A}
can be rewritten as
$$
\psi=\sum_{\la\in\GT_N}a_\la\varphi_\la.
$$
Initially, $\psi$ is a formal series, but, alternatively, the above sum can be
interpreted as the limit, in the $I$-adic topology, of the truncated finite
sums,
$$
\psi=\lim_{n_\pm\to\pm\infty}\;\sum_{\la\in\GT_N:\; n_+\ge\la_1,\; \la_N\ge
n_-}a_\la\varphi_\la.
$$

Therefore, one can say that the monomials $\varphi_\la$ form a homogeneous
\emph{topological} basis of $R$.

\subsection{Bases in $R$}

We are going to describe a general recipe for constructing various topological
bases in $R$ which are all consistent with the projective limit realization
$R=\varprojlim R(n_+,n_-)$.

Let us introduce a partial order on signatures: two signatures $\la$, $\mu$ may
be comparable only if they have the same length $N$, and then
$$
\la\ge\mu \,
\Leftrightarrow\,\la-\mu\in\Z_+(\epsi_1-\epsi_2)+\dots+\Z_+(\epsi_{N-1}-\epsi_N),
$$
where $\epsi_1,\dots,\epsi_N$ is the natural basis of the lattice $\Z^N$. In
particular, $\la\ge\mu$ implies $\sum\la_i=\sum\mu_i$. We write $\la>\mu$ if
$\la\ge\mu$ and $\la\ne\mu$. Note that the signatures of length $N$ are
precisely the highest weights of the irreducible representations of $U(N)$, and
the introduced order is nothing else than the standard \emph{dominance partial
order} on the set of weights of the reductive Lie algebra
$\mathfrak{gl}(N,\C)$, the complexified Lie algebra of $U(N)$.

We will be dealing with various symmetric \emph{Laurent} polynomials in several
variables $u_1,\dots,u_N$, $N=1,2,\dots$. The simplest example is the family of
monomial sums $m_\la$. Here $\la\in\GT_N$ and, by definition,
$$
m_\la=\sum_{(n_1,\dots,n_N)\in S(N)\cdot\la}u_1^{n_1}\dots u_N^{n_N},
$$
where $S(N)\cdot\la$ denotes the orbit of $\la$ under the action of the
symmetric group $S(N)$; in other words, the summation is over all
\emph{distinct} vectors $(n_1,\dots,n_N)\in\Z^N$ that can be obtained from
$(\la_1,\dots,\la_N)$ by permutations of the coordinates. By agreement,
$m_\varnothing:=1$ (the same agreement is tacitly adopted for other families
of polynomials that will appear below).

Assume we are given an arbitrary family $\{P_\la\}$ of homogeneous symmetric
Laurent polynomials indexed by signatures and satisfying the following
\emph{triangularity condition}:
\begin{equation}\label{eq2.B}
P_\la=\sum_{\mu:\,\mu\le\la}\al(\la,\mu) m_\mu,\qquad \al(\la,\mu)\in\C, \quad
\al(\la,\la)=1
\end{equation}
(examples will be given shortly). In particular, the number of variables in
$P_\la$ equals the length of $\la$.

With every such a family $\{P_\la\}$ we associate a family $\{\pi_\la\}$ of
homogeneous elements of $R$ in the following way. We form a generating series
for $\varphi_n$'s:
\begin{equation}\label{eq2.C}
\Phi(u):=\sum_{n\in\Z}\varphi_n u^n\in R[[u,u^{-1}]].
\end{equation}
Then the elements $\pi_\la$ in question are obtained as the coefficients in the
expansion
\begin{equation}\label{eq2.D}
\Phi(u_1)\dots\Phi(u_N)=\sum_{\la\in\GT_N}\pi_\la P_\la(u_1,\dots,u_N), \qquad N=1,2,\dots,
\end{equation}
and we agree that
$$
\pi_{\varnothing}=1.
$$

If $P_\la=m_\la$ for all $\la$, then the meaning of \eqref{eq2.D} is clear and
we obtain $\pi_\la=\varphi_\la$. But in the general case one has to explain how
to understand the sum in the right-hand side: the answer is that it converges
coefficient-wise, in the $I$-adic topology of $R$.

Here is an equivalent definition. The relation \eqref{eq2.D} is interpreted as
an infinite system of linear equations,
\begin{equation}\label{eq2.E}
\sum_{\la:\, \la\ge\mu}\al(\la,\mu)\pi_\la=\varphi_\mu, \qquad \forall\mu.
\end{equation}
The triangularity condition \eqref{eq2.B} gives a sense to the infinite sum in
the left-hand side of \eqref{eq2.E} and guarantees that the infinite matrix
$[\al(\la,\mu)]$ is invertible. Then we get
\begin{equation}\label{eq2.I}
\pi_\la=\sum_{\nu: \,\nu\ge\la}\be(\nu,\la)\varphi_\nu
\end{equation}
with some new coefficients $\be(\nu,\la)$ such that $\be(\la,\la)=1$.

It is evident that $\{\pi_\la\}$ is a topological basis in $R$. Moreover,
$\{\pi_\la\}$ is consistent with the ideals $I(n_+,n_-)$ meaning that
$I(n_+,n_-)$ is (topologically) spanned by the basis elements that are
contained in it, that is, by the elements $\pi_\la$, $\la\in\GT_N$, such that
$\la$ violates at least one of the inequalities $n_+\ge\la_1$, $\la_N\ge n_-$.
The quotient algebra $R(n_+,n_-)$ is, on the contrary, spanned by the
$\pi_\la$'s such that $\la$ satisfies the both inequalities.

\subsection{Example: the basis $\{\si_\la\}$ related to the Schur rational functions}

Let us turn now to concrete examples. The most important example is obtained
when as $\{P_\la\}$ we take the \emph{rational Schur functions} $s_\la$. These
are symmetric Laurent polynomials given by the same ratio-of-determinants
formula as the ordinary Schur polynomials, only the index $\la$ is an arbitrary
signature, so that the integers $\la_i$ are not necessarily nonnegative:
$$
s_\la(u_1,\dots,u_N)=\frac{\det[u_i^{\la_j+N-j}]}{V(u_1,\dots,u_N)},
$$
where the determinant in the numerator is of order $N$ and the denominator is
the Vandermonde,
$$
V(u_1,\dots,u_N)=\prod_{1\le i<j\le N}(u_i-u_j).
$$
The required triangularity condition \eqref{eq2.B} holds because $s_\la$ is an
irreducible character of $U(N)$. Another way to check \eqref{eq2.B} is to use
the combinatorial formula for the Schur polynomials.

Note that
\begin{equation}\label{eq2.J}
u_1\dots u_N s_\la(u_1,\dots,u_N)=s_{\la_1+1,\dots,\la_N+1}(u_1,\dots,u_N),
\end{equation}
which makes it possible to reduce many claims concerning the rational Schur
functions to the case of ordinary Schur polynomials.

For the basis $\{\pi_\la\}$ in $R$ corresponding to $P_\la=s_\la$ we use the
special notation $\{\si_\la\}$. Thus, the elements $\si_\la\in R$ are defined
as the coefficients of the expansion
\begin{equation}\label{eq2.H}
\Phi(u_1)\dots\Phi(u_N)=\sum_{\la\in\GT_N}\si_\la s_\la(u_1,\dots,u_N).
\end{equation}

Combining this with \eqref{eq2.C} and the definition of $s_\la$, one gets a
nice formula expressing $\si_\la$ through $\varphi_n$'s:
\begin{equation}\label{eq2.N}
\si_\la=\det[\varphi_{\la_i-i+j}]_{i,j=1}^N =\sum_{s\in
S(N)}\sgn(s)\varphi_{\la_1-1+s(1)}\dots\varphi_{\la_N-N+s(N)},
\end{equation}
where $S(N)$ denotes the group of permutations of\/ $\{1,\dots,N\}$ and\/
$\sgn(s)=\pm1$ is the sign of a permutation $s$.

Thus, the expansion of the elements of the basis $\{\si_\la\}$ in the basis
$\{\varphi_\nu\}$ has only finitely many nonzero terms. On the contrary, the
expansion of the elements of the latter basis in the former basis has
infinitely many terms (for $N\ge2$). For instance,
\begin{equation}\label{eq2.G}
\si_{\la_1,\la_2}=\varphi_{\la_1,\la_2}-\varphi_{\la_1+1,\la_2-1}
\end{equation}
but
\begin{equation}\label{eq2.F}
\varphi_{\la_1,\la_2}=\sum_{n=0}^\infty\si_{\la_1+n,\la_2-n}.
\end{equation}

\subsection{Example: bases related to Macdonald polynomials}

Observe that the Macdonald polynomials in finitely many variables (as well
their degeneration, the Jack polynomials) have a natural Laurent version,
because they satisfy the relation similar to \eqref{eq2.J}, see Macdonald
\cite[chapter VI, (4.17)]{Ma95}. Moreover, they satisfy the condition
\eqref{eq2.B}, see \cite[chapter VI, (4.7)]{Ma95}. Therefore, one may take
$P_\la(u_1,\dots,u_N)=P_\la(u_1,\dots,u_N; q,t)$ (the Laurent version of
Macdonald polynomials with two parameters $(q,t)$) or
$P_\la(u_1,\dots,u_N)=P^{(\al)}(u_1,\dots,u_N)$ (the Laurent version of Jack
polynomials with parameter $\al$), and then we get a certain topological basis
in $R$. In particular, the case $q=t$ gives the Schur polynomials and the basis
$\{\si_\la\}$, and the case $(q=0, t=1)$ give the monomial sums $m_\la$ and the
basis $\{\varphi_\la\}$.

\subsection{Structure constants of multiplication}

Let, as above,  $\{P_\la\}$ be a family of symmetric Laurent polynomials
satisfying the triangularity condition \eqref{eq2.B} and $\{\pi_\la\}$ be the
corresponding topological basis in $R$. Then any homogeneous element $\psi\in
R_N$ can be uniquely represented in the form
$\psi=\sum_{\la\in\GT_N}a_\la\pi_\la$ with some complex coefficients $a_\la$. I
am going to explain how to write the operation of multiplication in this
notation.

Let $M$ and $N$ be two nonnegative integers and $\la\in\GT_{M+N}$. Partitioning
the variables in $P_\la$ into two groups, of cardinality $M$ and $N$, we get an
expansion of the form
\begin{equation}\label{eq2.L}
P_\la(u_1,\dots,u_{M+N})=\sum_{\mu\in\GT_M,
\,\nu\in\GT_N}c(\la\mid\mu,\nu)P_\mu(u_1,\dots,u_M)P_\nu(u_{M+1},\dots,u_{M+N}),
\end{equation}
where $c(\la\mid\mu,\nu)$ are certain coefficients. Indeed, the existence,
finiteness, and uniqueness of this expansion is obvious in the case
$P_\la=m_\la$, and the general case is reduced to that case using the
triangularity property and the fact that for any signature $\la$, the set
$\{\mu: \mu\le\la\}$ is finite.

Now it follows from \eqref{eq2.D} that the same quantities $c(\la\mid\mu,\nu)$
are the structure constants of multiplication in the basis $\{\pi_\la\}$. That
is,
\begin{equation}\label{eq2.M}
\left(\sum a'_\mu\pi_\mu\right)\left(\sum a''_\nu\pi_\nu\right)=\sum
a_\la\pi_\la, \qquad  a_\la:=\sum_{\mu,\nu}c(\la\mid\mu,\nu)a'_\mu a''_\nu.
\end{equation}
The latter sum makes sense because we know that the expansion \eqref{eq2.L} is
finite.

\subsection{The isomorphism  $R\to \Rep(\mathfrak{gl}(2\infty))$}\label{sect2.B}

The remark below is based on a comment by Igor Frenkel.

Let $\mathfrak{gl}(\infty)$ denote the Lie algebra of complex matrices of
format $\infty\times\infty$ and finitely many nonzero entries. It has a natural
basis formed by the matrix units $E_{ij}$ with indices $i,j$ ranging over
$\{1,2,\dots\}$. The Schur-Weyl duality establishes a bijective correspondence
$S_\la\leftrightarrow V_\la$ between the irreducible representations of various
symmetric groups and a certain class of irreducible highest weight
$\mathfrak{gl}(\infty)$-modules. Here $\la=(\la_1,\la_2,\dots)$ is an arbitrary
partition, $S_\la$ is the corresponding irreducible $S(N)$-module (where
$N=|\la|:=\sum\la_i$), and $V_\la$ is the irreducible polynomial
$\mathfrak{gl}(\infty)$-module whose highest weight is $(\la_1,\la_2,\dots)$
with respect to the Borel subalgebra spanned by the $E_{ij}$ with $i\le j$.
Under the Schur-Weyl correspondence, the multiplication in
$\Rep(S(1),S(2),\dots)$ turns into the the tensor product of
$\mathfrak{gl}(\infty)$-modules. In this sense the algebra
$\Rep(S(1),S(2),\dots)=\Sym$ can be identified with
$\Rep(\mathfrak{gl}(\infty))$, the representation ring of polynomial
$\mathfrak{gl}(\infty)$-modules.

A similar interpretation exists for the algebra $R$. Namely, we replace
$\mathfrak{gl}(\infty)$ with its relative $\mathfrak{gl}(2\infty)$ --- the
latter Lie algebra has the basis $\{E_{ij}\}$ of matrix units with indices
$i,j$ ranging over $\Z$. Instead of the Schur-Weyl duality we use a version of
the ``fermion'' Howe duality \cite{Howe} between various unitary groups $U(N)$
and the Lie algebra $\mathfrak{gl}(2\infty)$. This duality establishes a
different kind of correspondence of representations, $T_\la\leftrightarrow
V_\la$, where $\la$ ranges over the set of all signatures. Here, for
$\la\in\GT_N$, we denote by $T_\la$ the corresponding irreducible
representation of $U(N)$, while $V_\la$ now stands for the irreducible
$\mathfrak{gl}(2\infty)$-module with highest weight
$\wh\la=(\wh\la_i)_{i\in\Z}$ which is described as follows.

Recall that every signature $\la$ of length $N$ can be represented as a pair
$(\la^+,\la^-)$ of two partitions (=Young diagrams) such that $\ell(\la^+)+\ell(\la^-)\le N$,
where $\ell(\,\cdot\,)$ is the conventional notation for the number of nonzero
parts of a partition. Namely,
$$
\la=(\la^+_1,\dots,\la^+_{\ell(\la^+)},
0,\dots,0,-\la^-_{\ell(\la^-)},\dots,-\la^-_1).
$$

In this notation, the weight correspondence $\la\to\wh\la$ looks as follows
$$
\wh\la_i=(\la^+)'_i, \quad i=1,2,\dots;  \quad \wh\la_{-(i-1)}=N-(\la^-)'_i, \quad
i=1,2,\dots,
$$
where  $(\la^\pm)'$ denotes the conjugate to $\la^\pm$  partition (=Young diagram).

Note that the coordinates $\wh\la_i$, $i\in\Z$, weakly decrease; the fact that
$\wh\la_0\ge\wh\la_1$ is equivalent to the inequality
$\ell(\la^+)+\ell(\la^-)\le N$ mentioned above.

About this instance of Howe duality see also Olshanski \cite[Section
2]{Ols-HolExt} and \cite[Section 17]{Ols-GordonBreach}.

As in the case of the Schur-Weyl duality, the multiplication in $R$
corresponds, on the Lie algebra side, to the tensor product of modules, so that
we get an isomorphism $R\to\Rep(\mathfrak{gl}(2\infty))$, where
$\Rep(\mathfrak{gl}(2\infty))$ is our notation for the representation ring for
a special class of $\mathfrak{gl}(2\infty)$-modules. This class is generated by
the weight modules that are locally nilpotent with respect to the upper
triangular subalgebra and such that, for every weight
$\wh\mu=(\wh\mu_i)_{i\in\Z}$, the coordinates $\wh\mu_i$ are nonnegative
integers which stabilize to a nonnegative integer $N$ as $i\to-\infty$ and to
$0$ as $i\to+\infty$. The irreducible modules
$V_\la\in\Rep(\mathfrak{gl}(2\infty)$ correspond to the basis elements
$\si_\la\in R$.

\subsection{Comparison of $R$ with $\Sym$}

The two algebras have both similarities and differences. The homogeneous
components in $\Sym$ have finite dimension while those in $R$ are not. The
latter fact seems to be the most evident difference between $R$ and $\Sym$. On
the other hand, both algebras are projective limits of polynomial algebras:
\begin{equation}\label{eq2.K}
\Sym=\varprojlim\C[e_1,\dots,e_n], \qquad
R=\varprojlim\C[\varphi_{n_-},\dots,\varphi_{n_+}].
\end{equation}
These polynomial algebras can be viewed as \emph{truncations} of the initial
algebras.

All familiar homogeneous bases in $\Sym$ are parameterized by partitions, and
those in $R$ are parameterized by signatures, which are relatives of
partitions. However, these two kinds of labels, partitions and signatures, are
related to the grading in a very different way: the degree of a basis element in
$\Sym$ is given by the sum of parts of the corresponding partition, while the
degree in $R$ corresponds to the length $N$ of a signature $\la$.

This is also seen from the comparison of the representation rings
$\Rep(\mathfrak{gl}(\infty))$ and $\Rep(\mathfrak{gl}(2\infty))$. As abstract
Lie algebras, $\mathfrak{gl}(\infty)$ and $\mathfrak{gl}(2\infty)$ are
isomorphic, but the respective classes of modules are different, and the
degrees of the irreducible modules are defined in a very different way.

Truncation in $\Sym$ and $R$ is also defined differently. Namely, a basis
element in $\Sym$ is not contained in the kernel of the truncation map
$\Sym\to\C[e_1,\dots,e_n]$ if and only if the length of the corresponding
partition does not exceed $n$, while truncation in $R$ is controlled by the
first and last coordinates of a signature $\la$. In the case when $\la_1>0$ and
$\la_N<0$, one has $\la_1=\ell((\la^+)')$ and $|\la_N|=\ell((\la^-)')$.

To define a homomorphism of the algebra $\Sym$ in a commutative algebra $A$
(for instance, an algebra of functions on a space) it suffices to specialize,
in an arbitrary way, the images of the generators $e_1, e_2,\dots$\,. In the
case of $R$, the situation is more delicate. Although the elements $\varphi_n$
play the role similar to that of the $e_n$'s, to define a morphism $R\to A$ it
does not suffice to specialize the image of the $\varphi_n$'s. The reason is
that these elements are not generators of $R$ in the purely algebraic
sense, but only \emph{topological} generators. It may well happen that a given
specialization of the $\varphi_n$'s can be extended only to a suitable
subalgebra of $R$. Two examples of subalgebras are examined below.

\subsection{The subalgebras $\RR$ and $\RR^0$}

For $\la\in\GT_N$, let
$$
\Dim_N\la:=s_\la(1,\dots,1).
$$
This is the dimension of the irreducible representation of $U(N)$ with highest
weight $\la$. As is well known (Weyl \cite{Weyl}, Zhelobenko \cite{Zhe})
\begin{equation}\label{eq2.R}
\Dim_N\la=\prod_{1\le i<j\le N}\frac{\la_i-\la_j-i+j}{j-i}\,.
\end{equation}

For a homogeneous element $\psi=\sum_{\la\in\GT_N}a_\la\si_\la\in R_N$ we
define its \emph{norm} (which may be infinite) by
\begin{equation*}
\Vert\psi\Vert:=\sup_{\la\in\GT_N}\frac{|a_\la|}{\Dim_N\la}\in\R_+\cup\{+\infty\}
\end{equation*}
and we extend this definition to non-homogeneous elements by setting
\begin{equation}\label{eq2.S}
\left\Vert\sum_{N=0}^M\psi_N\right\Vert:=\sum_{N=0}^M \Vert\psi_N\Vert, \qquad
\psi_N\in R_N, \quad N=0,1,\dots,M
\end{equation}
with the understanding that $\Vert1\Vert=1$.

\begin{definition}\label{def2.B}
We define $\RR\subset R$ as the subspace of elements with finite norm.
Obviously, $\RR$ is graded, so that we may write $\RR=\sum_{N=0}^\infty\RR_N$.
\end{definition}

\begin{proposition}\label{prop2.E}
$\RR$ is a normed algebra.
\end{proposition}

\begin{proof}
We have to prove that for any elements $\psi', \psi''\in R$ one has
\begin{equation}\label{eq2.Q}
\Vert\psi'\psi''\Vert\le \Vert\psi'\Vert\Vert\psi''\Vert.
\end{equation}

Indeed, assume first that $\psi'$ and $\psi''$ are homogeneous of degree $M$
and $N$, respectively, and write $\psi'=\sum a'_\mu\si_\mu$, $\psi''=\sum
a''_\nu\si_\nu$. By \eqref{eq2.M}
$$
\Vert\psi'\psi''\Vert
=\sup_{\la\in\GT_{M+N}}\frac{|\sum_{\mu,\nu}c(\la\mid\mu,\nu)a'_\mu
a''_\nu|}{\Dim\la}.
$$
Note that in our case, when $P_\la=s_\la$, the structure constants describe
the expansion of irreducible characters restricted from $U(M+N)$ to $U(M)\times
U(N)$. It follows that these constants are nonnegative integers. Next, by
counting dimensions one gets
\begin{equation*}
\sum_{\mu\in\GT_M,\,\nu\in\GT_N}c(\la\mid\mu,\nu)\Dim_M\mu\Dim_N\nu=\Dim_{M+N}\la.
\end{equation*}
Therefore, for every $\la\in\GT_{M+N}$,
$$
\frac{|\sum_{\mu,\nu}c(\la\mid\mu,\nu)a'_\mu a''_\nu|}{\Dim_{M+N}\la}
\le\Vert\psi'\Vert \Vert\psi''\Vert
\frac{\sum_{\mu,\nu}c(\la\mid\mu,\nu)\Dim_M\mu\Dim_N\nu}{\Dim_{M+N}\la}
=\Vert\psi'\Vert \Vert\psi''\Vert.
$$

This proves the desired inequality \eqref{eq2.Q}.

Now the general case, when $\psi'$ and $\psi''$ are not necessarily
homogeneous, follows immediately, by taking into account the definition of the
norm for non-homogeneous elements, \eqref{eq2.S}.
\end{proof}

\begin{definition}\label{def2.A}
For $N=1,2,\dots$ we define $\RR^0_N\subset\RR_N$ as the subspace of those
elements $\psi=\sum_{\la\in\GT_N}a_\la\si_\la\in R_N$ for which the ratio
$|a_\la|/\Dim_N\la$ tends to 0 as $\la$ goes to infinity. In other words, for
every $\epsi>0$ there should exist a finite subset of $\GT_N$ outside of which
$|a_\la|/\Dim_N\la\le\epsi$.

Next, we set
$$
\RR^0:=\bigoplus_{N=1}^\infty\RR^0_N
$$
and observe that $\RR^0$ is a norm-closed subspace of $\RR$.
\end{definition}

Let $R^{\operatorname{fin}}$ denote the space of finite linear combinations of
the basis elements $\si_\la$, where $\la\ne\varnothing$. By the very definition
of $\RR^0$, it coincides with the norm closure of $R^{\operatorname{fin}}$.

\begin{proposition}\label{prop2.F}
$\RR^0$ is closed under multiplication and so is a subalgebra in $\RR$.
\end{proposition}

Note that, according to our definition, $\RR^0$ does not contain the unity
element $1=\si_\varnothing$.

\begin{proof}
\emph{Step} 1. For any fixed $\mu\in\GT_M$ and $\nu\in\GT_N$, where $M,N\ge1$,
there exists a constant $C(\mu,\nu)$ such that
$$
\text{$c(\la\mid\mu,\nu)\le C(\mu,\nu)$ for all $\la\in\GT_{M+N}$}.
$$
This is a nontrivial claim whose proof is postponed to Section \ref{sect10}.

\emph{Step} 2. Let us fix $\mu$ and $\nu$ as above. We claim that
$$
\si_\mu\si_\nu\in\RR^0_{M+N}.
$$

Indeed, by the definition of the multiplication in $R$,
$$
\si_\mu\si_\nu=\sum_{\la\in\GT_{M+N}}c(\la\mid\mu,\nu)\si_\la.
$$
By the result of Step 1, the coefficients $c(\la\mid\mu,\nu)$ are bounded from
above. Therefore, to conclude that $\si_\mu\si_\nu\in\RR^0_{M+N}$ it remains to
show that $\Dim_N\la$ tends to infinity as $\la$ goes to infinity along the
subset
$$
X:=\{\la\in\GT_{M+N}: c(\la\mid\mu,\nu)>0\}.
$$

Observe that $\la\in X$ implies that the quantity $\la_1+\dots+\la_{M+N}$
remains fixed, because it is equal to $(\mu_1+\dots+\mu_M)+(\nu_1+\dots+\nu_N)$.

Therefore, as $\la$ goes to infinity along $X$, the difference
$\la_1-\la_{M+N}$ tends to $+\infty$, so that $\Dim_N\to\infty$, as it is seen
from Weyl's dimension formula \eqref{eq2.R}.

\emph{Step} 3. Let us show that $\RR^0$ is closed under multiplication. By the
result of Step 2, $R^{\operatorname{fin}}R^{\operatorname{fin}}$ is contained
in $\RR^0$. Since $R^{\operatorname{fin}}\subset\RR^0$ is dense with respect to
the norm topology, we conclude that $\RR^0\RR^0\subset\RR^0$.
\end{proof}

\subsection{Remarks on comultiplication}

By Frobenius' reciprocity,
\begin{equation*}
\operatorname{Ind}^{U(M+N)}_{U(M)\times U(N)} s_\mu\otimes
s_\nu=\sum_{\la\in\GT_{M+N}} c(\la\mid \mu,\nu)s_\la,
\end{equation*}
where the left-hand side is the induced character. So, one
could identify the formal symbols $\si_\la$ with the irreducible characters
$s_\la$ and say that the multiplication $R_M\otimes R_N\to R_{M+N}$ mimics the
operation of induction from $U(M)\times U(N)$ to $U(M+N)$. The reason to use
the separate notation $\si_\la$ is that characters should be viewed as
\emph{functions} while elements of $R$ behave as \emph{measures} (or, more
generally, distributions), which are dual objects with respect to functions.

Of course, on a finite or compact group, one can use the normalized Haar
measure $m_{\textrm{Haar}}$ to turn a function $f$ into a measure, $f
m_{\textrm{Haar}}$. However, one should not forget that functions and measures
have different functorial properties, so that when we restrict a character
$\chi$ to a subgroup, we regard $\chi$ as a function, while if we induct $\chi$
from a subgroup, we tacitly treat $\chi$ as a measure. In the case of finite
groups, the assignment $f\mapsto f m_{\textrm{Haar}}$ is a linear isomorphism
between the space of functions and the space of measures. Because of this,
$\Rep(S(1),S(2),\dots)$ (the representation ring of the family of symmetric
groups) possesses two dual operations, multiplication and comultiplication
making it a selfdual Hopf algebra (Zelevinsky \cite{Zelevinsky}). For compact Lie groups $U(N)$, the situation is more delicate as the space of
measures is much larger than the space of functions. This explains why the
representation ring $R$, as we have defined it, is not a Hopf algebra.

Note that one can use the same structure constants $c(\la \mid\mu,\nu)$ (in the
basis $\{\si_\la\}$) to construct a coalgebra $R^\circ$ which is paired with
$R$. Namely, a generic element of $R^\circ$ is a possibly infinite sum of
homogeneous elements which in turn are finite linear combinations of symbols
that we denote as $\chi_\la$; the comultiplication in $R^\circ$ is defined by
setting, for $\la\in\GT_N$,
$$
\triangle \chi_\la=\sum_{N_1,N_2:\, N_1+N_2=N}\,\sum_{\mu\in\GT_{N_1},
\nu\in\GT_{N_2}}c(\la\mid \mu,\nu)\,\chi_\mu\otimes\chi_\nu.
$$
Then the pairing $R\times R^\circ\to\C$ is defined in a natural way, by
proclaming $\{\si_\la\}$ and $\{\chi_\la\}$ to be biorthogonal systems.

Likewise, one can also define a suitable coalgebra $\RR\,^\circ$ which is
paired with the algebra $\RR$. However, in contrast to the case of the
representation ring for the symmetric groups, I do not see any way to modify the
definition of $R$ so that it becomes a selfdual Hopf algebra. Fortunately, for
our purposes we do not need to have both operations, multiplication and
comultiplication, to be defined on the same object.

\section{Characters of $U(\infty)$}\label{sect3}

Here we study a relationship between the representation ring $R$ and the dual
object $\Om=\wh{U(\infty)}$. In the symmetric group case, there is a
homomorphism of the algebra $\Sym$ into the algebra of continuous functions on
the dual object $\wh{S(\infty)}$, and the kernel of that homomorphism is the
principal ideal of $\Sym$ generated by $e_1-1$. The  purpose of this section is
to understand whether there exists something similar for the algebra $R$ and
the dual object $\Om$.

We exhibit three homomorphisms.

First, $R$ can be mapped into an algebra of functions defined on a certain subset
$\Om^0\subset\Om$ ($\Om^0$ is composed from some finite-dimensional ``faces''
of $\Om$). This map is far from being the desired analog but it is useful for
some technical purposes.

Second, the subalgebra $\RR$ can be mapped into $C(\Om)$, the Banach algebra of
bounded continuous functions on $\Om$.

Third, the above map sends the subalgebra $\RR^0\subset\RR$ into the subalgebra
$C_0(\Om)\subset C(\Om)$ formed by continuous functions vanishing at
infinity. The space $C_0(\Om)$ is of special interest for us because our main
objects of study, the generators of Markov processes on $\Om$, are operators on
the Banach space $C_0(\Om)$.

\subsection{Description of extreme characters: the Edrei-Voiculescu
theorem}\label{sect3.A}

For every $N=1,2,\dots$, we identify $U(N)$ with the subgroup of the group
$U(N+1)$ fixing the last basis vector in $\C^{N+1}$. This makes it possible  to
define the inductive limit group $U(\infty)=\varinjlim U(N)$. In other words,
elements of $U(\infty)$ are infinite unitary matrices $[U_{ij}]_{i,j=1}^\infty$
such that $U_{ij}=\de_{ij}$ when $i$ or $j$ is large enough.

We endow $U(\infty)$ with the inductive limit topology, which plainly means
that a function $f: U(\infty)\to\C$ is continuous if and only if for every $N$,
the function $f_N:=f\big|_{U(N)}$ is continuous on $U(N)$.

Notice that $f$ is a class function (respectively, a positive definite
function) if and only if so is $f_N$ for every $N$.

\begin{definition}\label{def3.A}
(i) By a \emph{character} of $U(\infty)$ we mean a continuous class function
$f:U(\infty)\to\C$ which is positive definite and normalized by $f(e)=1$.

(ii) Note that the set of all characters in the sense of (i) is a convex set.
Its extreme points are called \emph{extreme} or \emph{indecomposable}
characters.
\end{definition}

The extreme characters of $U(\infty)$ are analogs of the \emph{normalized}
irreducible characters
\begin{equation}\label{eq3.B}
\frac{s_\la(u_1,\dots,u_N)}{\Dim_N\la}, \qquad \la\in\GT_N.
\end{equation}

To describe the extreme characters we need to introduce some notation.

Let $\R_+\subset\R$ denote the set of nonnegative real numbers, $\R_+^\infty$
denote the product of countably many copies of $\R_+$, and set
$$
\R_+^{4\infty+2}=\R_+^\infty\times\R_+^\infty\times\R_+^\infty\times\R_+^\infty
\times\R_+\times\R_+.
$$
Let $\Om\subset\R_+^{4\infty+2}$ be the subset of sextuples
$$
\om=(\al^+,\be^+;\al^-,\be^-;\de^+,\de^-)
$$
such that
\begin{gather*}
\al^\pm=(\al_1^\pm\ge\al_2^\pm\ge\dots\ge 0)\in\R_+^\infty,\quad
\be^\pm=(\be_1^\pm\ge\be_2^\pm\ge\dots\ge 0)\in\R_+^\infty,\\
\sum_{i=1}^\infty(\al_i^\pm+\be_i^\pm)\le\de^\pm, \quad \be_1^++\be_1^-\le 1.
\end{gather*}
We observe that $\Om$ is a locally compact space in the topology inherited from
the product topology of $\R_+^{4\infty+2}$.

Instead of $\de^\pm$ it is often convenient to use the quantities
$$
\ga^\pm:=\de^\pm-\sum_{i=1}^\infty(\al_i^\pm+\be_i^\pm).
$$
Obviously, $\ga^+$ and $\ga^-$ are nonnegative. But, in contrast to $\de^+$ and
$\de^-$, they are \emph{not} continuous functions of $\om\in\Om$.

For $u\in\C^*$ and $\om\in\Om$ set
\begin{equation}\label{eqF.11}
\Phi(u;\om)= e^{\ga^+(u-1)+\ga^-(u^{-1}-1)}
\prod_{i=1}^\infty\frac{1+\be_i^+(u-1)}{1-\al_i^+(u-1)}
\,\frac{1+\be_i^-(u^{-1}-1)}{1-\al_i^-(u^{-1}-1)}.
\end{equation}

For any fixed $\om$, this is a meromorphic function in variable $u\in\C^*$ with
possible poles on $(0,1)\cup(1,+\infty)$. The poles do not accumulate to $1$,
so that the function is holomorphic in a neighborhood of the unit circle $\T:=\{u\in\C: |u|=1\}$.

Note that every conjugacy class of $U(\infty)$ contains a diagonal matrix with
diagonal entries $u_1,u_2, \ldots\in\T$, where only finitely many of $u_n$'s
are distinct from 1. These numbers are defined uniquely, within a permutation.
Thus every class function on $U(\infty)$ can be interpreted as a symmetric
function $\Psi(u_1,u_2,\dots)$.

\begin{theorem}[Edrei-Voiculescu]\label{thm3.A}
The extreme characters of the group $U(\infty)$ are precisely the functions of
the form
\begin{equation}\label{eq3.A}
\Psi_\om(u_1,u_2,\dots):=\prod_{k=1}^\infty\Phi(u_k;\om),
\end{equation}
where $\om$ ranges over\/ $\Om$.
\end{theorem}

Note that the product actually terminates because $\Phi(1;\om)=1$ and $u_k=1$
for $k$ large enough. As compared with the normalized irreducible characters of the groups $U(N)$ given by \eqref{eq3.B},
the extreme characters of $U(\infty)$ seem to be both more elementary and more sophisticated
objects. They are more elementary because they are given by a product formula,
but they are also more sophisticated as they depend on countably many
continuous parameters.

About various proofs and different facets of this fundamental theorem see Edrei
\cite{Edrei}, Voiculescu \cite{Vo76}, Boyer \cite{Boyer}, Vershik-Kerov
\cite{VK82}, Okounkov-Olshanski \cite{OO-Jack}, Borodin-Olshanski
\cite{BO-GT-Appr}, Petrov \cite{Petrov-MMJ}.

\begin{proposition}\label{prop3.C}
Given $\om\in\Om$, write the Laurent expansion of the function $u\mapsto
\Phi(u;\om)$ as
$$
\Phi(u;\om)=\sum_{n\in\Z}\wh\varphi_n(\om) u^n.
$$

For $n\in\Z$ fixed, the coefficient $\wh\varphi_n(\om)$ is a continuous
function on $\Om$ vanishing at infinity.
\end{proposition}

\begin{proof}
See Borodin-Olshanski \cite[Proposition 2.10]{BO-GT-Appr}.
\end{proof}

Recall that we denoted by $\Phi(u)$ the formal generating series assembling the
variables $\varphi_n$, see \eqref{eq2.C} above. The fact that we employ now a
similar notation is not occasional. As explained below, the functions
$\wh\varphi_n(\om)$ serve as the image of the generators $\varphi_n\in R$ under
the maps mentioned in the preamble to the section.

\subsection{The quotient algebra $\wh R=R/J$}

Observe that $\Phi(1;\om)\equiv1$, which implies
\begin{equation}\label{eq3.H}
\sum_{n\in\Z}\wh\varphi_n(\om)=1, \qquad \om\in\Om.
\end{equation}
This relation motivates the following definitions.

Let us set
$$
\varphi:=\sum_{n\in\Z}\varphi_n.
$$
and let $J:=(\varphi-1)\subset R$ be the principal ideal  generated by the
element $\varphi-1$. The ideal $J$ and the quotient algebra $\wh R:=R/J$ play
an important role in our theory, similar to that of the ideal
$(e_1-1)\subset\Sym$ and the quotient ring $\Sym/(e_1-1)$ in Vershik-Kerov's
theory \cite{VK-RingTheorem}, \cite{VK-GordonBreach}.

The quotient ring $\wh R$ is a filtered algebra: its filtration is inherited
from the filtration in $R$, which in turn is determined from the grading; the
latter is not inherited because the ideal $J$ is not homogeneous.

We will prove a few simple propositions concerning the algebra $\wh R$.

\begin{proposition}\label{prop3.D}
For every $N=0,1,2,\dots$, the intersection $J\cap R_N$ is trivial.
\end{proposition}

\begin{proof}
This is a formal consequence of the fact that $R$ has no zero divisors (which
in turn follows from the isomorphism $R=\varprojlim R(n_+,n_-)$).

Indeed, assume $\psi\in J\cap R_N$ and show that $\psi=0$. There exists
$\psi'\in R$ such that $\psi=(\varphi-1)\psi'$. Since $R$ has no zero divisors,
the degree of $\psi'$ cannot be larger than $N-1$, so one can write
$$
\psi'=\psi_0+\dots+\psi_{N-1}, \qquad \psi_i\in R_i.
$$
Then
$$
\psi=\sum_{i=0}^{N-1}(\varphi-1)\psi_i
=-\psi_0+(\varphi\psi_0-\psi_1)+\dots+(\varphi\psi_{N-2}-\psi_{N-1})+\varphi\psi_{N-1}.
$$
Since $\psi$ is homogeneous of degree $N$, we have $\psi=\varphi\psi_{N-1}$ and
$$
-\psi_0=(\varphi\psi_0-\psi_1)=\dots=(\varphi\psi_{N-2}-\psi_{N-1})=0.
$$
This implies $\psi_0=\dots=\psi_{N-1}=0$ and finally $\psi=0$.
\end{proof}

Let, as above, $n_+\ge n_-$ be a couple of integers. We denote by $J(n_+,n_-)$
the ideal in $R$ generated by the ideals $J$ and $I(n_+,n_-)$. Under the
homomorphism $R\to R(n_+,n_-)$, the image of $J$ is the principal ideal
generated by the element $(\varphi_{n_-}+\dots+\varphi_{n_+})-1$. We set
\begin{equation}\label{eq3.C}
\wh R(n_+,n_-):=R/J(n_+,n_-).
\end{equation}
This algebra can be identified with the quotient
$$
\C[\varphi_{n_-},\dots,\varphi_{n_+}]\big/(\varphi_{n_-}+\dots+\varphi_{n_+}-1)
$$
and so is isomorphic to the algebra of polynomials with $n_+-n_-$ variables.

\begin{proposition}\label{prop3.E}
As $n_\pm\to\pm\infty$, the intersection of the kernels of the composite
homomorphisms
$$
R\to R(n_+,n_-)\to\wh R(n_+,n_-)
$$
coincides with $J$.
\end{proposition}

\begin{proof}
This is a trivial consequence of the absence of zero divisors. Indeed, the
ideal $J$ lies in the intersection of the kernels in question. Conversely,
assume $\psi\in R$ belongs to the intersections of the kernels and show that
$\psi\in J$, that is, there exists $\psi'\in R$ such that
$\psi=(\varphi-1)\psi'$.

By the assumption, for every couple $(n_+,n_-)$ there exists an element
$\psi'_{n_+,n_-}\in R(n_+,n_-)$ such that the image of $\psi$ in $R(n_+,n_-)$
is equal to
$$
(\varphi_{n_-}+\dots+\varphi_{n_+}-1)\psi'_{n_+,n_-}.
$$
Note that this element is unique and its degree is bounded from above by
$\deg(\psi)-1$.

It follows that there exists an element $\psi'=\varprojlim \psi'_{n_+,n_-}$.
The elements $\psi$ and $(\varphi-1)\psi'$ have the same image under the map
$R\to R(n_+,n_-)$, for every $(n_+,n_-)$. Therefore, these elements are equal
to each other.
\end{proof}

\begin{corollary}\label{cor3.A}
The algebra $\wh R$ can be identified with the projective limit of filtered
algebras $\wh R(n_+,n_-)$ as $n_\pm\to\pm\infty$.
\end{corollary}

\begin{proof}
Since $R=\varprojlim R(n_+,n_-)$, there is a natural homomorphism $\wh R\to
\varprojlim \wh R(n_+,n_-)$.  Proposition \ref{prop3.E} shows that  it is
injective. Let us check that it is also surjective. Without loss of generality
one can assume that $n_+>0>n_-$. Then we use the relation
$\varphi_{n_-}+\dots+\varphi_{n_+}=1$ in $\wh R(n_+,n_-)$ to eliminate
$\varphi_0$ and to lift $\wh R(n_+,n_-)$ into $R(n_+,n_-)$ as the subalgebra
$R'(n_+,n_-)$ generated by $\varphi_{n_-}, \dots, \varphi_{-1},
\varphi_1,\dots,\varphi_{n_+}$. This makes it possible to identify $\varprojlim
\wh R(n_+,n_-)$ with $\varprojlim R'(n_+,n_-)$, where both limits are taken in
the category of filtered algebras. Then the surjectivity in question becomes
obvious.
\end{proof}

We say that two signatures $\mu\in\GT_N$ and $\la\in\GT_{N+1}$ \emph{interlace} if 
\begin{equation}\label{eq3.AA}
\la_i\ge\mu_i\ge\la_{i+1}, \qquad i=1,\dots,N,
\end{equation}
and then we write $\mu\prec\la$ or, equivalently, $\la\succ\mu$. 
By agreement, any signature $\la\in\GT_1$ is interlaced with the empty signature $\varnothing\in\GT_0$.  

\begin{proposition}\label{prop3.AA}
For any $\mu\in\GT_N$, where $N=0,1,2,\dots$, one has
\begin{equation}\label{eq3.AB}
\varphi\si_\mu=\sum_{\la\in\GT_{N+1}:\,\la\succ\mu}\si_\la.
\end{equation}
\end{proposition}

\begin{proof}
The classical Gelfand--Tsetlin branching rule says that for $\la\in\GT_{N+1}$,
$$
s_\la(u_1,\dots,u_{N+1})=\sum_{\mu\in\GT_N:\, \mu\prec\la}s_\mu(u_1,\dots,u_N) u_{N+1}^{|\la|-|\mu|},
$$
where $|\la|:=\sum\la_i$, $|\mu|:=\sum\mu_j$. This gives us the structure constants $c(\la\mid\mu,\nu)$ (see Section \ref{sect2}) for the basis of rational Schur functions in the special case when $\la\in\GT_{N+1}$, $\mu\in\GT_N$, and $\nu=n\in\GT_1=\Z$. Namely, in this special case,
$$
c(\la\mid\mu,\nu)=\begin{cases} 1, & \text{if $\mu\prec\la$ and $n=|\la|-|\mu|$}, \\
0, & \text{otherwise.}
\end{cases}
$$
Combining this with the definition of the multiplication in $R$ we get \eqref{eq3.AB}.
\end{proof}

The proposition shows that the ideal $J$ coincides with the closed linear span of the elements of the form
$$
-\si_\mu+\sum_{\la:\,\la\succ\mu}\si_\la,
$$
where $\mu$ ranges over the set\/ $\GT_0\cup\GT_1\cup\dots$ of all signatures.
This fact is used below in the proof of Proposition \ref{prop7.C}.

\subsection{The simplices $\Om(n_+,n_-)$}\label{sect3.B}

Let $(n_+, n_-)$ be a couple of integers such that $n_+\ge0\ge n_-$. We set
\begin{multline*}
\Om(n_+,n_-):=\big\{\om=(\al^\pm,\be^\pm,\de^\pm)\; : \;\al^\pm_i=0 \quad
\textrm{for all $i$}, \quad  \be^+_i=0 \quad \textrm{for $i>n_+$},\\  \be^-_j=0
\quad \textrm{for $j>|n_-|$}, \quad \de^+=\be^+_1+\dots+\be^+_{n_+}, \quad
\de^-=\be^-_1+\dots+\be^-_{|n_-|} \big\} \subset \Om
\end{multline*}
This a compact subset of $\Om$ whose elements depend only on $n_++|n_-|$
independent parameters $\be^+_1,\dots,\be^+_{n_+}, \be^-_1\dots,\be^-_{|n_-|}$.
The conditions on these parameters can be written in the form
$$
1-\be^+_{|n^+|}\ge\dots\ge1-\be^+_1\ge\be^-_1\ge\dots\ge\be^-_{n^-}\ge0
$$
(because $\be^+_1+\be^-_1\le1$). This means that $\Om(n_+,n_-)$ can be viewed
as a simplex of dimension $n_++|n_-|$.

If $\om\in\Om(n_+,n_-)$, then the function $\Phi(u;\om)$ drastically simplifies
and takes the form
$$
\Phi(u;\om)=\prod_{i=1}^{n_+}(1-\be^+_i+\be^+_i u)\cdot
\prod_{j=1}^{|n_-|}(1-\be^-_j+\be^-_j u^{-1}).
$$
The function $\varphi_n(\om)$ vanishes identically on $\Om(n_+,n_-)$ unless
$n_+\ge n\ge n_-$.

Let $C(\Om(n_+,n_-))$ denote the algebra of continuous functions on the simplex
$\Om(n_+,n_-)$. By Proposition \ref{prop3.C}, every function
$\wh\varphi_n(\om)$ is continuous on $\Om(n_+,n_-)$.

Recall that $J(n_+,n_-)$ denotes the principal ideal in $R(n_+,n_-)$ generated
by the element $(\varphi_{n_-}+\dots+\varphi_{n_+})-1$.

\begin{proposition}\label{prop3.B}
The kernel of the homomorphism
$$
R(n_+,n_-)=\C[\varphi_{n_-},\dots,\varphi_{n_+}]\to C(\Om(n_+,n_-))
$$
assigning to $\varphi_n$ the function $\wh\varphi_n(\om)$ on $\Om(n_+,n_-)$
coincides with the ideal $J(n_+,n_-)$.
\end{proposition}

\begin{proof}
Since $\wh\varphi_n(\om)$ vanishes on $\Om(n_+,n_-)$ unless $n_+\ge n\ge n_-$,
the equality \eqref{eq3.H} shows that
$$
\sum_{n=n_-}^{n_+}\wh\varphi_n\big|_{C(\Om(n_+,n_-))}=1.
$$
It remains to prove that this is the only relation.

Let us examine the special case when $n_-=0$. To simplify the notation, set
$n_+=m$ and
$$
(t_1,\dots,t_m):=(1-\be^+_m,\dots,1-\be^+_1)
$$
Let us write $\wh\varphi_n(t_1,\dots,t_m)$ instead of $\wh\varphi_n(\om)$,
where $n=0,\dots,m$. These are symmetric polynomials in $t_1,\dots,t_m$
satisfying
$$
\prod_{i=1}^m(t_i+(1-t_i)u)=\sum_{n=0}^m\wh\varphi_n(t_1,\dots,t_m)u^n.
$$
For instance, for $m=2$,
$$
\wh\varphi_0(t_1,t_2)=t_1t_2, \quad \wh\varphi_1(t_1,t_2)=(t_1+t_2)-2t_1t_2,
\quad \wh\varphi_2(t_1,t_2)=(1-t_1)(1-t_2).
$$

In the case under consideration, the claim of the proposition is equivalent to
saying that the only algebraic relation between these $m+1$ polynomials is that
their sum equals $1$. Let us prove the last assertion.

Evidently, our polynomials lie in the linear span of the elementary symmetric
polynomials $e_n(t_1,\dots,t_m)$, where $n=0,\dots,m$ and $e_0:=1$. Therefore,
it suffices to check that our polynomials are linearly independent.

To do this, we evaluate them in the following $m+1$ points of $\R^m$:
$$
x_k:=(\,\underbrace{1,\dots,1}_{m-k}, \underbrace{0,\dots,0}_{k}\,), \qquad
k=0,\dots,m.
$$
At $x_k$, the product $\prod(t_i+(1-t_i)u)$ equals $u^k$. This implies that
$\wh\varphi_n(x_k)=\de_{nk}$, which concludes the proof in our special case.

Finally, the case $n_-<0$ is readily reduced to the special case $n_-=0$ by
using the twisting transformation $\tau$ defined in the next subsection.
\end{proof}

Proposition \ref{prop3.B} shows that the quotient ring $\wh R(n_+,n_-)=
R(n_+,n_-)/J(n_+,n_-)$ is embedded into the algebra $C(\Om(n_+,n_-))$ of
continuous functions on the simplex $\Om(n_+,n_-)$ as the subalgebra of
polynomial functions.

Together with Proposition \ref{prop3.E} this makes it possible to realize the
quotient ring $\wh R= R/J$ as an algebra of functions on the subset
\begin{equation}\label{eq3.I}
\Om^0:=\bigcup_{n_+\ge n_-}\Om(n_+,n_-)\subset\Om.
\end{equation}

\subsection{Symmetries}

There exist natural transformations of characters of $U(\infty)$, which
preserve the subset of extreme characters and thus induce transformations (or
\emph{symmetries}) $\Om\to\Om$ of the parameter space.

One such transformation is the operation of \emph{conjugation} mapping a
character $f(U)$ to the conjugate character $\overline{f(U)}$ (here $U$ ranges
over $U(\infty)$). Conjugation induces the symmetry $\om\mapsto\om^*$ of $\Om$
consisting in switching
$(\al^+,\be^+,\de^+)\leftrightarrow(\al^-,\be^-,\de^-)$.

Another kind of transformation is the multiplication of $f(U)$ by $\det(U)$. In
terms of the eigenvalues this amounts to multiplication by the product
$u_1u_2\dots$. The corresponding symmetry of $\Om$ leaves the parameters
$\al^\pm$ intact and changes the remaining parameters in the following way:
\begin{gather*}
(\be^+_1,\be^+_2,\dots)\mapsto (1-\be^-_1,\be^+_1,\be^+_2,\dots)\\
(\be^-_1,\be^-_2,\dots)\mapsto (\be^-_2,\be^-_3,\dots)\\
\de^+\mapsto\de^++(1-\be^-_1)\\
\de^-\mapsto\de^--\be^-_1.
\end{gather*}
Note that $1-\be^-_1\ge \be^+_1$ because of the condition
$\be^+_1+\be^-_1\le1$.

We call this the \emph{twisting} symmetry of $\Om$ and denote it as $\om\mapsto
\tau(\om)$. Obviously, $\tau$ is invertible.

Under the symmetry $\om\mapsto\om^*$, the subset $\Om(n_+,n_-)$ is mapped onto
$\Om(-n_-,-n_+)$. If $n_-\le-1$, then the twisting symmetry $\tau$ maps
$\Om(n_+,n_-)$ onto $\Om(n_++1,n_-+1)$.

Recall that so far we assumed $n_+\ge0\ge n_-$. However, one can extend the
definition of $\Om(n_+,n_-)$ so that the equality
$\tau(\Om(n_+,n_-))=\Om(n_++1,n_-+1)$ will be valid for every couple $n_+\ge
n_-$, dropping the assumption that $n_+\ge0$ and $n_-\le0$. For instance, if
$n_-\ge1$, then the first $n_-$ coordinates in $\be^+$ are equal to 1 and the
actual parameters are $\be^+_{n_-+1},\dots,\be^+_{n_+}$.

\subsection{The homomorphisms $\RR\to C(\Om)$ and $\RR^0\to C_0(\Om)$}

Recall that the functions $\wh\varphi_n(\om)$ introduced in Proposition
\ref{prop3.C} belong to the Banach space $C_0(\Om)$. At this moment we only
exploit the fact that they belong to $C(\Om)$. Let us assign to every generator
$\varphi_n\in R$ the function $\wh\varphi_n(\om)$. We are going to extend this
correspondence to a norm continuous homomorphism $\RR\to C(\Om)$.

Let us start by assigning to every basis element $\si_\la$ a suitable function
$\wh\si(\om)$. This can be done in two equivalent ways.

\emph{First way}. We use the determinantal formula \eqref{eq2.N} and set for
$\la\in\GT_N$ and $\om\in\Om$
\begin{equation}\label{eq3.E}
\wh\si_\la(\om):=\det[\wh\varphi_{\la_i-i+j}(\om)].
\end{equation}

\emph{Second way}. Restricting the extreme character $\Psi_\om$ defined in
\eqref{eq3.A} to the subgroup $U(N)\subset U(\infty)$ gives us a normalized positive
definite class function on $U(N)$, which can be expanded into an absolutely and
uniformly convergent series on the irreducible characters of $U(N)$. Then the
desired quantities $\wh\si_\la(\om)$ arise as the coefficients of this
expansion. Passing to matrix eigenvalues one can write this in the form
\begin{equation}\label{eq3.J}
\Phi(u_1;\om)\dots\Phi(u_N;\om)=\sum_{\la\in\GT_N}\wh\si_\la(\om)
s_\la(u_1,\dots,u_N).
\end{equation}

{}From \eqref{eq3.E} it follows that the functions $\wh\si_\la(\om)$ belong to
$C(\Om)$ (even to $C_0(\Om)$), and from \eqref{eq3.J} we see that
$\wh\si_\la(\om)\ge0$ (because the function in the left-hand side is positive
definite). This is an important observation which will be exploited below.

Here is one more useful consequence of \eqref{eq3.J}: setting $u_1=\dots=u_N=1$
we get the identity
\begin{equation}\label{eq2.O}
\sum_{\la\in\GT_N}\Dim_N\la\,\wh\si_\la(\om)=1.
\end{equation}

Next, given an element $\psi=\sum a_\la\si_\la\in\RR$, we want to assign to it the
function $\wh\psi(\om)=\sum a_\la\wh\si_\la(\om)$ on $\Om$.

\begin{proposition}\label{prop3.A}
{\rm(i)} For every element $\psi=\sum a_\la\si_\la\in\RR$, the series
$\wh\psi(\om):=\sum a_\la\wh\si_\la(\om)$ converges absolutely at every point
$\om\in\Om$. Moreover, the resulting function on $\Om$ is bounded and its
supremum norm does not exceed $\Vert\psi\Vert$.

{\rm(ii)} The map $\psi\mapsto\wh\psi(\,\cdot\,)$ is an algebra homomorphism
$\RR\to C(\Om)$.

{\rm(iii)} The kernel of this homomorphism is the principal ideal $\J\subset
\RR$ generated by the element $\varphi-1$. This ideal coincides with
$J\cap\RR$.
\end{proposition}

\begin{proof}
\emph{Step} 1. Let us check (i). We will assume first that $\psi$ is
homogeneous of degree $N$. Then we have (recall that $\wh\si_\la(\om)\ge0$)
\begin{equation}\label{eq3.D}
\sum_\la|a_\la|\wh\si_\la(\om)=\sum_\la\frac{|a_\la|}{\Dim_N\la}\Dim_N\la\,\wh\si_\la(\om)
\le\Vert\psi\Vert \sum_\la\Dim_N\la\,\wh\si_\la(\om)=\Vert\psi\Vert,
\end{equation}
where the final equality follows from \eqref{eq2.O}.

The same holds for arbitrary (not necessarily homogeneous) elements, by the
very definition of the norm in $\RR$.

\emph{Step} 2. Let us check that the map $\psi\mapsto\wh\psi(\,\cdot\,)$ is
consistent with multiplication. That is, for any two elements
$\psi',\psi''\in\RR$ and any $\om\in\Om$ one has
$$
\wh{\psi'}(\om)\wh{\psi''}(\om)=\wh\psi(\om), \qquad \psi:=\psi'\psi''.
$$

Indeed, without loss of generality we may assume that $\psi'$ and $\psi''$ are
homogeneous, of degree $M$ and $N$, respectively. Write
$$
\psi'=\sum_{\mu\in\GT_M}a'_\mu\si_\mu, \qquad
\psi''=\sum_{\nu\in\GT_N}a''_\nu\si_\nu, \qquad
\psi=\sum_{\la\in\GT_{M+N}}a_\la\si_\la.
$$
By virtue of \eqref{eq2.M}, we have
$$
a_\la=\sum_{\mu,\nu}c(\la\mid\mu,\nu)a'_\mu a''_\nu,
$$
where the structure constants correspond to the choice $P_\la=s_\la$.

It readily follows that the desired statement is reduced to the following
identity:  for any fixed $\mu\in\GT_M$ and $\nu\in\GT_N$ one has
\begin{equation}\label{eq2.P}
\wh\si_\mu(\om)\wh\si_\nu(\om)=\sum_{\la\in\GT_{M+N}}c(\la\mid\mu,\nu)\wh\si_\la(\om),
\qquad \om\in\Om.
\end{equation}

This identity, in turn, follows from the second definition of the quantities 
$\wh\si_\la(\om)$ (formula \eqref{eq3.J} above) and the identity
$$
s_\la(u_1,\dots,u_{M+N})=
\sum_{\mu\in\GT_M,\,\nu\in\GT_N}c(\la\mid\mu,\nu)s_\mu(u_1,\dots,u_M)
s_\nu(u_{M+1},\dots,u_{M+N}).
$$
Necessary interchanges of the order of summation are justified because all the
series are absolutely convergent.

\emph{Step 3}. Let us show that the functions $\wh\psi(\om)$ are continuous on
$\Om$. We may assume that $\psi$ is homogeneous of degree $N$. Then the
corresponding function $\wh\psi(\om)$ is given by the series
$\sum_{\la\in\GT_N}a_\la\wh\si_\la(\om)$. We know that the functions
$\wh\si_\la(\om)$ are continuous, but one cannot immediately conclude that
$\wh\psi$ is also continuous because the series is not necessarily convergent
in the norm topology of $C(\Om)$. This difficulty is resolved in the following
way. Since the space $\Om$ is locally compact, it suffices to prove that the
series for $\wh\psi$ converges uniformly on compact subsets of $\Om$. Looking
at \eqref{eq3.D} one sees that it suffices to do this for the series
$\sum_\la\Dim_N\la\,\wh\si_\la(\om)$. By \eqref{eq2.O}, it converges to the
constant function $1$ at every point $\om\in\Om$. Since all the summands are
nonnegative, the convergence is uniform on compact sets, as desired.

Thus, we completed the proof of (ii).

\emph{Step} 4. Obviously, the element $\varphi$ belongs to $\RR$, so that the
principal ideal $\J\subset\RR$ generated by $\varphi-1$ is well defined. Let us
show that $\J=J\cap\RR$. To do this we have to check that if $\psi\in R$ is
such that $(\varphi-1)\psi\in\RR$, then $\psi\in\RR$. This is proved by the
same argument as in the proof of Proposition \ref{prop3.D}.

\emph{Step} 5. Finally, let us check that $\J$ coincides with the kernel of the
homomorphism $\psi\mapsto\wh\psi(\,\cdot\,)$. We know that the function
$\wh\varphi(\om)$ is the constant function $1$, so $\J$ is contained in the
kernel.

It remains to show that if, conversely, $\psi\in\RR$ is such that
$\wh\psi(\om)\equiv0$ on $\Om$, then $\psi\in\J$. Here we apply the result
stated at the very end of Subsection \ref{sect3.B}. It suffices to use the fact
that the function $\wh\psi(\om)$ vanishes on $\Om^0$. Then that result says
that $\psi\in J$. Because $J\cap\RR=\J$, we conclude that $\psi\in\J$.
\end{proof}

\begin{corollary}\label{cor3.B}
The homomorphism of Proposition \ref{prop3.A} determines by restriction a
homomorphism $\RR^0\to C_0(\Om)$.
\end{corollary}

\begin{proof}
By the definition of the subalgebra $\RR^0\subset\RR$, the linear span of the
basis elements $\si_\la$ is dense in $\RR^0$ with respect to the norm topology.
On the other hand, as it was pointed above, the functions $\wh\si_\la(\om)$
belong to $C_0(\Om)$. Since $C_0(\Om)$ is closed in $C(\Om)$ and the
homomorphism $\RR\to C(\Om)$ is norm continuous, this shows that the image of
the whole subalgebra $\RR^0$ is contained in $C_0(\Om)$.
\end{proof}

\subsection{Analog of the Vershik-Kerov ring theorem}

Let $\RR_+\subset\RR$ denote the closed (in the norm topology) convex cone
spanned by the elements $\si_\la$. For two elements $\psi_1, \psi_2\in\RR_+$,
write $\psi_1\le\psi_2$ if $\psi_2-\psi_1\in\RR_+$.

The following result is similar to the so-called \emph{ring theorem} due to
Vershik and Kerov, see \cite[Theorem 6]{VK-RingTheorem} and \cite[Introduction,
Theorem 4]{Kerov-book}.

\begin{proposition}
{\rm(i)} The set of characters of $U(\infty)$ in the sense of Definition
\ref{def3.A} is in a natural one-to-one correspondence with linear functionals
$F:\RR\to\C$ satisfying the following properties{\rm:}

\begin{itemize}

\item $F$ is norm-continuous and takes real nonnegative values on the cone
$\RR_+$.

\item If $\psi\in\RR_+$ is the least upper bound for a  sequence
$0\le\psi_1\le\psi_2\le\dots$, then $F(\psi)=\lim_{n\to\infty} F(\psi_n)$.

\item $F(1)=1$ and $F(\varphi\psi)=F(\psi)$ for every $\psi\in\R$.

\end{itemize}

{\rm(ii)} A character is extreme if and only if the corresponding functional
$F$ is multiplicative, that is, $F(\psi_1\psi_2)=F(\psi_1)F(\psi_2)$ for any
$\psi_1,\psi_2\in\RR$.

\end{proposition}

The proof is similar to that given in \cite{VK-RingTheorem} (see also a more
detailed version in Gnedin-Olshanski \cite[Section 8.7]{GO-zigzag}).

This result does not depend on the classification of the extreme characters and
provides one more proof of their multiplicativity.

\section{The operator $\mathbb D_\zw$}\label{sect4}

\begin{definition}\label{def4.A}
Fix an arbitrary quadruple $(z,z',w,w')$ of complex parameters and introduce
the following formal differential operator in countably many variables
$\{\varphi_n:n\in\Z\}$
$$
\D_\zw=\sum_{n\in\Z}A_{nn}\frac{\pd^2}{\pd\varphi_n^2}+2\sum_{\substack{n_1,n_2\in\Z\\
n_1>n_2}} A_{n_1 n_2}\frac{\pd^2}{\pd\varphi_{n_1}\pd\varphi_{n_2}}
+\sum_{n\in\Z}B_n\frac{\pd}{\pd\varphi_n}, \
$$
where, for any indices $n_1\ge n_2$,
\begin{equation}\label{eq4.B}
\begin{gathered}
A_{n_1 n_2}=\sum_{p=0}^\infty(n_1-n_2+2p+1)(\varphi_{n_1+p+1}\varphi_{n_2-p}
+\varphi_{n_1+p}\varphi_{n_2-p-1})\\
-(n_1-n_2)\varphi_{n_1}\varphi_{n_2}
-2\sum_{p=1}^\infty(n_1-n_2+2p)\varphi_{n_1+p}\varphi_{n_2-p}
\end{gathered}
\end{equation}
and, for any $n\in\Z$,
\begin{equation}\label{eq4.C}
\begin{gathered}
B_n=(n+w+1)(n+w'+1)\varphi_{n+1}+(n-z-1)(n-z'-1)\varphi_{n-1}\\
-\bigl((n-z)(n-z')+(n+w)(n+w')\bigr)\varphi_n.
\end{gathered}
\end{equation}
\end{definition}

\medskip

Note that only coefficients $B_n$ depend on the parameters $(z,z',w,w')$.

\begin{proposition}
The operator $\D_\zw$ is correctly defined on $R$.
\end{proposition}

Note that not every formal differential operator in variables $\varphi_n$ can
act on $R$. Here is a very simple example: application of
$\sum_{n\in\Z}\frac{\pd}{\pd\varphi_n}$ to the element
$\varphi=\sum_{n\in\Z}\varphi_n$ gives the meaningless expression
$\sum_{n\in\Z}1$. As is seen from the argument below, the validity of the
proposition relies on the concrete form of the coefficients of $\D_\zw$.

\begin{proof}

(i) Obviously, when $\D_\zw$ is formally applied to a monomial in $R$, the
result is a well-defined element of $R$. We have to prove that, more generally,
the same holds when $\D_\zw$ is applied to any homogeneous element $g\in R$. In
other words, the infinite sum arising in $\D_\zw g$ cannot contain infinitely
many nonzero terms proportional to one and the same monomial.

(ii) Given a monomial $\varphi_\la=\varphi_{\la_1}\dots\varphi_{\la_N}$ indexed
by a signature $\la$, define its \emph{support} $\supp \varphi_\la$ as the
lattice interval $[a,b]:=\{a,\dots,b\}\subset\Z$, where
$a=\la_N=\min(\la_1,\dots,\la_N)$ and $b=\la_1=\max(\la_1,\dots,\la_N)$.

{}From \eqref{eq4.B} is is evident that for every monomial $\varphi_\mu$
entering
$$
A_{n_1 n_2}\frac{\pd^2 \varphi_\la}{\pd\varphi_{n_1}\pd\varphi_{n_2}},
$$
one has $\supp \varphi_\mu\supseteq[a,b]$.

(iii) Likewise, from \eqref{eq4.B} it is clear that if a monomial $\varphi_\mu$
enters
$$
B_n\frac{\pd \varphi_\la}{\pd\varphi_n}
$$
and $[a',b']:=\supp \varphi_\mu$, then one has $|a'-a|\le1$, $|b'-b|\le1$.

(iv) Let again, as in (i) above, $g$ be a homogeneous element of $R$, and
examine the infinite sum $\D_\zw g$ resulting from application of $\D_\zw$ to
$g$. Observe that there exist only finitely many monomials of a prescribed
degree and with the support contained in a prescribed lattice interval.
Therefore, (ii) and (iii) guarantee that the undesired accumulation of
infinitely many proportional terms in $\D_\zw g$ is excluded.

\end{proof}

\begin{proposition}\label{prop4.C}
If $z=n_+$ and\/ $w=-n_-$, where $n_+\ge n_-$ are integers, then the operator\/
$\D_\zw$ preserves the ideal $I(n_+,n_-)$ and hence correctly determines an
operator acting on the quotient ring $R(n_+,n_-)=R/I(n_+,n_-)$.

\end{proposition}

\begin{proof}
The ideal $I(n_+,n_-)$ consists of (possibly infinite) linear combinations of
monomials whose support is not contained in the lattice interval $[n_-,n_+]$.
Step (ii) of the argument above shows that the application of the second order
terms in $\D_\zw$ enlarges the supports and so preserves the ideal
$I(n_+,n_-)$. Note that this holds for any values of the parameters.

Now let us examine the effect of the application of a first degree term
$B_n\frac{\pd}{\pd \varphi_n}$. {}From \eqref{eq4.C} it is seen that the only
danger may come from the quantities
$$
(n+w+1)(n+w'+1)\varphi_{n+1}\big|_{n=n_--1}, \quad
(n-z-1)(n-z'-1)\varphi_{n-1}\big|_{n=n_++1}.
$$
But these quantities vanish because, by our assumption, $w=-n_-$ and $z=n_+$.
\end{proof}

\begin{proposition}\label{prop4.A}
For any fixed integer $m$, the operator\/ $\D_\zw$ is invariant under the change
of variables $\varphi_n\mapsto \varphi_{n+m}$ {\rm($n\in\Z$)} combined with the
shift of parameters
$$
z\to z+m, \quad z'\to z'+m, \quad w\to w-m, \quad w'\to w-m.
$$
\end{proposition}

In connection with this proposition see also Remark 3.7 in \cite{BO-AnnMath}.

\begin{proof}
Indeed, the indicated simultaneous shift of the variables and parameters does
not change the coefficients $A_{n_1n_2}$ and $B_n$.
\end{proof}

The next proposition is not so evident:

\begin{proposition}\label{prop4.B}
The operator $\D_\zw$ preserves the principal ideal $J\subset R$.
\end{proposition}

\begin{proof}
We will prove that $\D_\zw$ commutes with the operator of multiplication by $\varphi$, which obviously implies that $\D_\zw$ preserves $J$.

Take an arbitrary element $F\in R$ and observe that  $\D_\zw(\varphi F)-\varphi\D_\zw F$ equals 
$$
2\sum_{n\in\Z}\left(A_{nn}+\sum_{n_1:\,
n_1>n}A_{n_1n}+\sum_{n_2:\,n_2<n}A_{nn_2}\right)\frac{\pd
F}{\pd\varphi_n}+\left(\sum_{n\in\Z}B_n\right)F.
$$

We are going to check that this expression vanishes. More precisely, the $n$th
summand in the first sum vanishes for every $n\in\Z$ and the sum $\sum B_n$
vanishes, too.

Indeed, by \eqref{eq4.C}, one has
\begin{multline*}
\sum_{n\in\Z}B_n=
\sum_{n\in\Z}(n+w+1)(n+w'+1)\varphi_{n+1}+\sum_{n\in\Z}(n-z-1)(n-z'-1)\varphi_{n-1}\\
-\sum_{n\in\Z}\bigl((n-z)(n-z')+(n+w)(n+w')\bigr)\varphi_n.
\end{multline*}
By making the change $n\to n\pm1$ in the first two sums one sees that that the
whole expression equals $0$.

Next, let us check that
$$
A_{nn}+\sum_{n_1:\, n_1>n}A_{n_1n}+\sum_{n_2:\,n_2<n}A_{nn_2}=0.
$$
By virtue of Proposition \ref{prop4.A}, it suffices to do this for the
particular value $n=0$, which slightly simplifies the notation. Then the
identity in question can be written as
\begin{equation}\label{eq4.A}
A_{00}+\sum_{m>0}A_{m0}+\sum_{m>0}A_{0,-m}=0.
\end{equation}

Let us write down explicitly all the summands:
$$
A_{00}=\sum_{p\ge0}(2p+1)[\varphi_{p+1}\varphi_{-p}+\varphi_p\varphi_{-p-1}]
-2\sum_{p\ge1}2p\varphi_p\varphi_{-p}.
$$

$$
A_{m0}=\sum_{p\ge0}(m+2p+1)[\varphi_{m+p+1}\varphi_{-p}+\varphi_{m+p}\varphi_{-p-1}]
-m\varphi_m\varphi_0-2\sum_{p\ge1}(m+2p)\varphi_{m+p}\varphi_{-p}.
$$

$$
A_{0,-m}=\sum_{p\ge0}(m+2p+1)[\varphi_{p+1}\varphi_{-m-p}+\varphi_{p}\varphi_{-m-p-1}]
-m\varphi_0\varphi_{-m}-2\sum_{p\ge1}(m+2p)\varphi_{p}\varphi_{-m-p}.
$$
Then a slightly tedious but direct examination shows that in \eqref{eq4.A}, all
the terms are cancelled.
\end{proof}

\section{The method of intertwiners}\label{sect5}

This method was proposed in Borodin-Olshanski \cite{BO-GT-Dyn}. The method
allows one to construct Markov processes on dual objects to inductive limit
groups like $S(\infty)$ or $U(\infty)$ by essentially algebraic tools. Here we
describe its idea. For more details, see \cite{BO-GT-Dyn}, Borodin-Olshanski
\cite{BO-EJP}, and the expository paper Olshanski \cite{Ols-SPb}.

\subsection{Generalities on Markov kernels and Feller processes}

Let $X$ and $Y$ be two measurable spaces. Recall that a \emph{Markov kernel}
with source space $X$ and target space $Y$ is a function $P(x,A)$, where the
first argument $x$ ranges over $X$ and the second argument is a measurable
subset of $Y$; next, one assumes that the following two conditions hold (see
e.g. Meyer \cite{Mey66}):

\medskip

$\bullet$ For $A$ fixed, $P(\,\cdot\,, A)$ is a measurable function on $X$.

$\bullet$ For $x$ fixed, $P(x,\,\cdot\,)$ is a probability measure on $Y$ (we
will denote it by $P(x,dy)$).

\medskip

When the second space $Y$ is a discrete space, it is convenient to interpret
the kernel as a function on $X\times Y$ by setting $P(x,y):=P(x,\{y\})$. In the
case when both spaces are discrete, $P(x,y)$ is a stochastic matrix of format
$X\times Y$.

We regard a Markov kernel $P$ as a surrogate of map between $X$ and $Y$,
denoted as $P: X\dasharrow Y$ and called a \emph{link}. Here the dashed arrow
symbolizes the fact that a link is not an ordinary map: it assigns to a given
point $x\in X$ not a single point in $Y$ but a probability distribution on $Y$.

The superposition of two links $P':X\dasharrow Y$ and $P'': Y\dasharrow Z$ is
the link $P=P'P''$ between $X$ and $Z$ defined by
$$
P(x,dz)=\int_{y\in Y}P'(x,dy)P''(y,dz).
$$
If both $X$ and $Y$ are discrete, then the superposition becomes the matrix
product.

Every link $P:X\dasharrow Y$ induces a contractive linear operator $f\mapsto Pf$ from the
Banach space of bounded measurable functions on $Y$ to the similar function
space on $X$:
$$
(P f)(x)=\int_{y\in Y}P(x,dy)f(y), \qquad x\in X.
$$

Assuming $X$ and $Y$ are locally compact spaces, we say that $P: X\dasharrow Y$
is a \emph{Feller link} if the above operator maps $C_0(Y)$ into $C_0(X)$. Note
that the superposition of Feller links is a Feller link, too. (We recall that $C_0(X)$ consists of continuous functions on $X$ vanishing at infinity. If $X$ is a
discrete space, then the continuity assumption is trivial and $C_0(X)$ consists
of arbitrary functions vanishing at infinity.)

Now we recall a few basic notions from the theory of Markov processes (see
Liggett \cite{Liggett}, Ethier-Kurtz \cite{EK}).

A \emph{Feller semigroup} on a locally compact space $X$ is a strongly
continuous semigroup $P(t)$, $t\ge0$, of contractive operators on $C_0(X)$
given by Feller links $P(t;x, dy)$. A well-known abstract theorem says that a
Feller semigroup gives rise to a Markov process on $X$ with transition function
$P(t;x,dy)$. The processes derived from Feller semigroups are called
\emph{Feller processes}; they form a particularly nice subclass of general
Markov processes.

A Feller semigroup $P(t)$ is uniquely determined by its \emph{generator}. This
is a closed dissipative operator $A$ on $C_0(X)$ given by
$$
Af=\lim_{t\to+0}\frac{P(t)f-f}{t}.
$$
The \emph{domain} of $A$, denoted by $\dom A$, is the (algebraic) subspace
formed by those functions $f\in C_0(X)$ for which the above limit exists; $\dom
A$ is always a dense subspace. Every subspace $\F\subset\dom A$ for which the
closure of $A\big|_{\F}$ equals $A$ is called a \emph{core} of $A$. One can say
that a core is an ``essential domain'' for $A$. Very often, the full domain of
a generator is difficult to describe explicitly, and then one is satisfied by
exhibiting a core $\F$ with the explicit action of the generator on $\F$.

\subsection{Stochastic links between dual objects}

Here we introduce concrete examples of stochastic links we will dealing with.

For a compact group $G$, we denote by $\wh G$ the set of irreducible characters
of $G$ and call  it the \emph{dual object} to $G$. Given  $\chi\in\wh G$, we
denote by $\wt \chi$ the corresponding normalized character:
$$
\wt\chi(g)=\frac{\chi(g)}{\chi(e)}, \qquad g\in G.
$$

In the special case when $G$ is commutative, $\wt\chi=\chi$ and $\wh G$ is a
discrete group, but in the general case (when $G$ is noncommutative), the dual
object does not possess a group structure and we regard it simply as a discrete
space.

To every morphism $\iota: G_1\to G_2$ of compact groups there corresponds a
\emph{canonical} ``dual'' link $\La: \wh G_2\dasharrow \wh G_1$, defined as follows.
For every irreducible character $\chi\in\wh G_2$, its superposition with
$\iota$ is a finite linear combination of irreducible characters $\chi'\in\wh
G_1$ with nonnegative integral coefficients. It follows that the superposition
of $\wt\chi$ with $\iota$ is a convex linear combination of normalized
irreducible characters of the group $G_1$; the coefficients of the latter
expansion are just the entries of the stochastic matrix $\La$. That is,
$$
\wt\chi(\iota(g))=\sum_{\chi'\in \wh G_1}\La(\chi,\chi')\wt{\chi'}(g), \qquad
g\in G_1, \quad \chi\in\wh G_2.
$$

If $G_1\to G_2$ and $G_2\to G_3$ are two morphisms of compact groups, then it
is evident that the superposition of the canonical dual links $\wh
G_3\dasharrow \wh G_2$ and $\wh G_2\dasharrow \wh G_1$ coincides with the
canonical link $\wh G_3\dasharrow \wh G_1$ corresponding to the composition
morphism $G_1\to G_3$.

Consider now the infinite chain of groups
$$
U(1)\subset U(2)\subset U(3)\subset\dots
$$
as defined in the beginning of Subsection \ref{sect3.A}. For every $N<M$, this
chain defines an embedding $U(N)\hookrightarrow U(M)$, and we denote by
$\La^M_N(\la,\mu):\GT_M\dasharrow\GT_N$ the corresponding dual link, which is a
stochastic matrix of format $\GT_M\times\GT_N$. In particular, for $M=N+1$ this
matrix takes the form
\begin{equation}\label{eq5.C}
\La^{N+1}_N(\la,\mu)=\begin{cases}\dfrac{\Dim_N\mu}{\Dim_{N+1}\la}, & \text{if
$\mu\prec\la$}\\ 0, &\text{otherwise},\end{cases}
\end{equation}
where $\mu\prec\la$ means that the two signatures \emph{interlace} in the sense
that
$$
\la_i\ge\mu_i\ge\la_{i+1}, \qquad i=1,\dots,N,
$$
see Borodin-Olshanski \cite[Section 1.1]{BO-GT-Dyn} for more details.

Next, consider the embedding $U(N)\hookrightarrow U(\infty)$ (the image of the
former group in the latter group consists of the infinite unitary matrices
$[U_{ij}]$ such that $U_{ij}=\de_{ij}$ unless both $i$ and $j$ are less or
equal to $N$). We define the dual object $\wh{U(\infty)}$ as the set of extreme
characters and identify it with $\Om$. Then the above definition of the dual
link is still applicable with the extreme characters of $U(\infty)$ playing the
role of the (nonexisting) normalized irreducible characters. The resulting
Markov kernel $\Om\dasharrow \GT_N$ has the form
\begin{equation}\label{eq5.A}
\La^\infty_N(\om,\la)=\Dim_N\la\cdot \wh\si_\la(\om), \qquad \om\in\Om, \quad
\la\in\GT_N,
\end{equation}
where $\wh\si_\la(\om)$ is defined in Section \ref{sect3}. The derivation of
this formula is simple: by \eqref{eq3.A}, the restriction of the extreme
character $\Psi_\om$ to the subgroup $U(N)$ is given by the function
$\Phi(u_1;\om)\dots\Phi(u_N;\om)$; the expansion of that function on the
irreducible characters $\chi_\la=s_\la$ is given by \eqref{eq3.J}, and we only
need to introduce the factor $\Dim_N\la$ to get the required expansion on the
normalized characters $\wt\chi_\la=s_\la/\Dim_N\la$.

\begin{proposition}\label{prop5.A}
The canonical links $\La^M_N:\GT_M\dasharrow\GT_N$ and
$\La^\infty_N:\Om\dasharrow\GT_N$ are Feller links.
\end{proposition}

For a proof, see Borodin-Olshanski \cite[Corollary 2.11 and Proposition
2.12]{BO-GT-Appr}.

\subsection{The method of intertwiners}

Let $X$ and $Y$ be locally compact spaces, $P_X(t)$ and $P_Y(t)$ be Feller
semigroups on $X$ and $Y$, respectively, and $\La:X\dasharrow Y$ be a Feller
link. We say that $\La$ \emph{intertwines} the semigroups $P_X(t)$ and $P_Y(t)$
if the following commutation relation holds
$$
P_X(t)\La=\La P_Y(t), \qquad t\ge0.
$$
This relation can be understood as an equality of links or, equivalently, as an
equality of operators acting from $C_0(Y)$ to $C_0(X)$.

\begin{proposition}\label{prop5.B}
Assume we are given a family $\{P_N(t): N=1,2,3,\dots\}$ of Feller semigroups,
where the $N$th semigroup acts on $C_0(\GT_N)$. Further, assume that these
semigroups are intertwined by the canonical links $\La^{N+1}_N$, so that
$$
P_{N+1}(t)\La^{N+1}_N=\La^{N+1}_N P_N(t), \qquad N=1,2,3,\dots, \quad t\ge0.
$$

Then there exists a unique Feller semigroup $P_\infty(t)$ on $C_0(\Om)$
characterized by the property
$$
P_\infty(t)\La^\infty_N=\La^\infty_N P_N(t), \qquad N=1,2,\dots, \quad t\ge0.
$$
\end{proposition}

\begin{proof}
See Proposition 2.4 in Borodin-Olshanski \cite{BO-GT-Dyn}. The fact that the
hypothesis of this proposition is satisfied in our concrete situation is
established in Subsection 3.3 of that paper.
\end{proof}

\begin{proposition}\label{prop5.C}
We keep to the hypotheses of Proposition \ref{prop5.B}. Let $A_N$ and
$A_\infty$ denote the generators of the semigroups $P_N(t)$ and $P_\infty(t)$,
respectively.

{\rm(i)} For every $N=1,2,\dots$ and every $f\in\dom(A_N)$, the vector
$\La^\infty_Nf$ belongs to $\dom(A_\infty)$ and one has
$$
A_\infty\La^\infty_N f=\La^\infty_N A_N f.
$$

{\rm(ii)} Assume additionally that for each $N=1,2,3,\dots$ we are given a core
$\F_N\subseteq \dom(A_N)$ for the operator $A_N$. Then the linear span of the
vectors of the form $\La^\infty_N f$, where $N=1,2,\dots$ and $f\in\F_N$, is a
core for $A_\infty$.
\end{proposition}

\begin{proof}
Claim (i) directly follows from the definition of the generator. Claim (ii) is
established in Borodin-Olshanski \cite[Proposition 5.2]{BO-EJP}.

\end{proof}

\subsection{The degenerate case}\label{sect5.A}

Let us fix a couple of integers $n_+\ge n_-$ and set
\begin{equation}\label{eq5.B}
\GT_N(n_+,n_-)=\{\nu\in\GT_N: n_+\ge\nu_1\ge\dots\ge\nu_N\ge n_-\}.
\end{equation}
Note that this is a finite set.

If $\mu\in\GT_M(n_+,n_-)$ and $N<M$, then $\La^M_N(\mu,\nu)$ vanishes unless
$\nu\in\GT_N(n_+,n_-)$. So, $\La^M_N$ induces a link $\GT_M(n_+,n_-)\dasharrow
\GT_N(n_+,n_-)$. Likewise, if $\om\in\Om(n_+,n_-)$, then
$\La^\infty_N(\om,\nu)$ vanishes unless $\nu\in\GT_N(n_+,n_-)$. So,
$\La^\infty_N$ induces a link $\Om(n_+,n_-)\dasharrow \GT_N(n_+,n_-)$.

When $\GT_N$ (with $N=1,2,3,\dots$) and $\Om$ are replaced by $\GT_N(n_+,n_-)$
and $\Om(n_+,n_-)$, respectively, all the results of the present section remain
valid. The proofs are extended automatically, and we only point out some
simplifications:

In Proposition \ref{prop5.A}, the claim concerning the Feller property for the
links $\La^M_N$ becomes redundant as the links $\GT_M(n_+,n_-)\dasharrow
\GT_N(n_+,n_-)$ are finite matrices. Next, because $\Om(n_+,n_-)$ is a compact
space, the Feller property for the link $\La^\infty_N:\Om(n_+,n_-)\dasharrow
\GT_N(n_+,n_-)$ simply means that the functions of the form $\om\to
\La^\infty_N(\om,\nu)$ are continuous on $\Om(n_+,n_-)$.

In Proposition \ref{prop5.B}, one should replace $C_0(\Om)$ by
$C(\Om(n_+,n_-))$, the Banach space of all continuous functions on the compact
space $\Om(n_+,n_-)$.

In Proposition \ref{prop5.C}, because the sets $\GT_N(n_+,n_-)$ are finite, the
generators $A_N$ are finite-dimensional, so that $\dom(A_N)$ is the whole space
of functions on $\GT_N(n_+,n_-)$.

\section{Markov processes on $\Om$ and their generators}\label{sect6}

This section contains some necessary material from Borodin-Olshanski
\cite{BO-GT-Dyn}, together with a brief motivation. In that paper, we
constructed a family $\{X_\zw\}$ of continuous time Markov processes on the
space $\Om$, indexed by the quadruple of parameters $(\zw)$ ranging over a
certain subset of $\C^4$. The infinitesimal generator of $X_\zw$, denoted by
$A_\zw$, is an unbounded operator on the Banach space $C_0(\Om)$. The results
of \cite{BO-GT-Dyn} tell us how $A_\zw$ acts on a subspace $\wh\F\subset
C_0(\Om)$, the (algebraic) linear span of the functions $\wh\si_\la(\om)$,
where $\la$ ranges over the set of all signatures except $\la=\varnothing$. The
explicit formulas for this action are the starting point for the computations
in the remaining part of the paper. Note that $\wh\F$ serves as a core for
the generator $A_\zw$, so that it is uniquely determined by
its restriction to $\wh\F$.

\subsection{Special bilateral birth-death processes}

Birth-death processes form a well-studied class of continuous time Markov
chains. The state space of every birth-death process is the set $\Z_+$ of
nonnegative integers, and the process is determined by specifying the
quantities $q(n,n\pm1)$,  the \emph{jump rates} from state $n\in\Z_+$ to the
neighboring states $n\pm1$, with the understanding that $q(0,-1)=0$, which
prevents from leaving the subset $\Z_+\subset\Z$. Under appropriate constraints
on the jump rates the process is well defined (that is, does not explode,
meaning that, with probability 1, one cannot escape to infinity in finite
time).

The \emph{bilateral} birth-death processes are defined in a similar way, only
now the state space is the whole lattice $\Z$ and the jump rates $q(n,n\pm1)$
are assumed to be strictly positive for all $n\in\Z$. Again, one needs some
restrictions to be imposed on these quantities in order that the process be
non-exploding. Bilateral birth-death processes are not so widely known as the
ordinary ones. However, they were also discussed in the literature.

We are interested in bilateral birth-death processes whose jump rates
$q(n,n\pm1)$ are quadratic functions in variable $n$. We write them in the form
\begin{equation}\label{eq6.B}
q(n,n-1)=(w+n)(w'+n), \qquad q(n,n+1)=(z-n)(z'-n).
\end{equation}
It is readily verified that these quantities are strictly positive for all
$n\in\Z$ if and only if each of pairs $(z,z')$ and $(w,w')$ belongs to the
subset $\mathscr Z\subset\C^2$ defined by
\begin{multline}\label{eq6.A}
\mathscr Z:= \{(\zeta,\zeta')\in(\C\setminus\Z)^2\mid
\zeta'=\bar{\zeta}\}\\
 \cup
\{(\zeta,\zeta')\in(\R\setminus\Z)^2\mid m<\zeta,\zeta'<m+1 \text{ for some }
m\in\Z\}.
\end{multline}

Note that if $(\zeta,\zeta')\in\mathscr Z$, then $\zeta+\zeta'$ is real.

\begin{definition}\label{def6.A}
We say that a quadruple $(\zw)\in\C^4$ is \emph{admissible} if
$(z,z')\in\mathscr Z$, $(w,w')\in\mathscr Z$, and $z+z'+w+w'>-1$.
\end{definition}

\begin{proposition}\label{prop6.A}
Let $(\zw)\in\C^4$ be admissible.

{\rm(i)} There exists a non-exploding bilateral birth-process with the jump
rates given by \eqref{eq6.B}.

{\rm(ii)} This process is a Feller process.

{\rm(iii)} Its generator is implemented by the difference operator $D_\zw$ on
$\Z$ acting on functions $f(n)$, $n\in\Z$ by
\begin{multline}\label{eq6.C}
(D_\zw f)(n)=(z-n)(z'-n)(f(n+1)-f(n))\\
+(w+n)(w'+n)(f(n-1)-f(n)),
\end{multline}
and the domain of the generator consists of those functions $f\in C_0(\Z)$ for
which $D_\zw f\in C_0(\Z)$.

\end{proposition}

\begin{proof}
Statements (i) and (ii) are the subject of  Theorem 5.1 in Borodin-Olshanski
\cite{BO-GT-Dyn}, and (iii) is their formal consequence, as explained in
\cite[Proposition 4.6]{BO-GT-Dyn}.
\end{proof}

We refer to \cite{BO-GT-Dyn} for more details. Note that the property of
non-explosion is the same as \emph{regularity} of the so-called \emph{Q-matrix}
(or the \emph{matrix of jump rates}), see \cite[Section 4]{BO-GT-Dyn} and
references therein. In our case, this matrix is simply the matrix of the
difference operator $D_\zw$. This is a tridiagonal $\Z\times\Z$ matrix
$Q=[q(n,n')]$ with the entries $q(n,n\pm1)$ given by \eqref{eq6.B}, the
diagonal entries
$$
q(n, n)=-q(n,n+1)-q(n,n-1),
$$
and all remaining entries equal to 0.

\subsection{Feller dynamics on $\GT_N$}\label{sect6.A}

As explained in Borodin-Olshanski \cite[Section 5.2]{BO-GT-Dyn}, Proposition
\ref{prop6.A} admits an extension with $\Z$ replaced by $\GT_N$, where
$N=1,2,3,\dots$ (recall that $\GT_1=\Z$). To state it we need first to define a
matrix $Q=[q(\nu,\mu)]$ of format $\GT_N\times\GT_N$. It depends on $(\zw)$ and
has the following form:

$\bullet$ the entries $q(\nu,\mu)$ equal 0 unless $\mu=\nu$ or $\mu=\nu\pm
\epsi_i$, where $i=1,\dots,N$ and $\epsi_1,\dots,\epsi_N$ stands for the
canonical basis of $\Z^N$;

$\bullet$ the (nonzero) off-diagonal entries are given by
\begin{equation}\label{eq6.D}
q(\nu,\nu\pm \epsi_i)=\frac{\Dim_N(\nu\pm \epsi_i)}{\Dim_N\nu}\,r(\nu,\nu\pm
\epsi_i),
\end{equation}
where
\begin{equation}\label{eq6.E}
r(\nu,\nu+\epsi_i)=(z-\nu_i+i-1)(z'-\nu_i+i-1), \qquad i=1,\dots,N,
\end{equation}
and
\begin{equation}\label{eq6.F}
r(\nu,\nu-\epsi_i)=(w+\nu_i-i+N)(w'+\nu_i-i+N), \qquad i=1,\dots,N;
\end{equation}

$\bullet$ the diagonal entries are given by
\begin{multline}\label{eq6.G}
q(\nu,\nu)=-\sum_{\mu:\,
\mu\ne\nu}q(\nu,\mu)\\=(z+z'+w+w')\frac{N(N-1)}2+\frac{(2N-1)N(N-1)}3-\sum_{\mu:\,
\mu\ne\nu}r(\nu,\mu).
\end{multline}

When $N=1$, this agrees with the definition of the preceding subsection. (To compare the above formulas with those from \cite[Section 5.2]{BO-GT-Dyn}, take into account a shift of parameters indicated in \cite[(6.1) and (6.2)]{BO-GT-Dyn}.)

For $\nu=(\nu_1,\dots,\nu_N)\in\GT_N$, we set
$$
\nu^*=(-\nu_N,\dots,-\nu_1).
$$
The correspondence $\nu\mapsto\nu^*$ is an involutive bijection
$\GT_N\to\GT_N$.

\begin{proposition}\label{prop6.D}
One has
$$
q(\nu,\mu)=q^*(\nu^*,\mu^*),
$$
where the matrix $[q^*(\,\cdot\,,\,\cdot\,)]$ is obtained from the matrix
$[q^*(\,\cdot\,,\,\cdot\,)]$ by switching $(z,z')\leftrightarrow(w,w')$.
\end{proposition}

\begin{proof}
This is readily checked.
\end{proof}

\begin{proposition}\label{prop6.B}
Let $(\zw)\in\C^4$ be admissible in the sense of Definition \ref{def6.A}.

{\rm(i)} For every $N=1,2,3,\dots$, the\/ $\GT_N\times\GT_N$ matrix
$Q=[q(\nu,\mu)]$ defined above is regular, so that there exists a non-exploding
continuous time Markov process on $\GT_N$ with the jump rates given by the
off-diagonal entries $q(\nu,\mu)$.

{\rm(ii)} This process is a Feller process.

{\rm(iii)} Its generator is implemented by the $N$-variate difference operator
$D_{\zw\mid N}$ on $\GT_N\subset\Z^N$ acting on functions $f(\nu)$,
$\nu\in\GT_N$ by
\begin{equation}\label{eq6.H}
(D_{\zw\mid N} f)(\nu)=\sum_{\mu\in\GT_N}q(\nu,\mu)f(\mu)
=\sum_{\mu\in\GT_N\setminus\{\nu\}}q(\nu,\mu)(f(\mu)-f(\nu)),
\end{equation}
and the domain of the generator consists of those functions $f\in C_0(\GT_N)$
for which $D_{\zw\mid N} f\in C_0(\GT_N)$.
\end{proposition}

\begin{proof}
Statements (i) and (ii) are proved in  Borodin-Olshanski \cite[Theorem
5.4]{BO-GT-Dyn}, and (iii) is their formal consequence, as explained in
\cite[Proposition 4.6]{BO-GT-Dyn}.
\end{proof}

\begin{proposition}\label{prop6.F}
For any $(\zw)\in\C^4$ and any $N=0,1,2,\dots$ the following relation holds
\begin{equation}\label{eq6.O}
D_{\zw\mid N+1}\La^{N+1}_N=\La^{N+1}_N D_{\zw\mid N}\qquad \forall
N=1,2,\dots\,.
\end{equation}
\end{proposition}

(Recall that $\La^{N+1}_N: \GT_{N+1}\dasharrow\GT_N$ are the canonical links defined in \eqref{eq5.C}.)

\begin{proof}
In a slightly different notation, this is proved in \cite[Proposition
6.2]{BO-GT-Dyn}.
\end{proof}

This result serves as the basis for the construction described in the next
subsection. It is also used in Section \ref{sect9} below.

\subsection{Feller dynamics on $\Om$}

Throughout this subsection we assume, as before, that  $(\zw)$ is
admissible (Definition \ref{def6.A}).

\begin{proposition}\label{prop6.C}
For $N=1,2,\dots$, we denote by $P_{\zw\mid N}(t)$ the Feller semigroup on
$C_0(\GT_N)$ afforded by Proposition \ref{prop6.B}.

{\rm(i)} These semigroups $P_{\zw\mid N}(t)$ satisfy the hypothesis of
Proposition \ref{prop5.B}, that is, one has
$$
P_{\zw\mid N+1}(t)\La^{N+1}_N=\La^{N+1}_N P_{\zw\mid N}(t), \qquad t\ge0,
$$
for every $N=1,2,3,\dots$\,.

{\rm(ii)} There exists a unique Feller semigroup $P_{\zw\mid\infty}(t)$ on
$C_0(\Om)$ characterized by the property
$$
P_{\zw\mid\infty}(t)\La^\infty_N=\La^\infty_N P_N(t), \qquad N=1,2,\dots, \quad
t\ge0.
$$
\end{proposition}

(Recall that $\La^\infty_N:\Om\dasharrow \GT_N$ are the links defined in \eqref{eq5.A}.)

\begin{proof}
Claim (i) is established in Borodin-Olshanski \cite[theorem 6.1]{BO-GT-Dyn}.
Claim (ii) follows from Claim (i) by virtue of Proposition \ref{prop5.B}.
\end{proof}

\begin{definition}\label{def6.B}
In what follows $A_{\zw\mid N}$ denotes the generator of the semigroup $P_{\zw\mid N}(t)$ on $C_0(\GT_N)$
and $A_\zw$ denotes the generator of the semigroup $P_{\zw\mid\infty}(t)$ on $C_0(\Om)$.
\end{definition}

In the next proposition and its proof we use the quantities  $q(\nu,\mu)$ and
$r(\nu,\mu)$ that were defined in the preceding subsection. Note that they
depend on the parameters $\zw$, and $N$.

\begin{proposition}\label{prop6.E}
Let $N=1,2,\dots$\,. For every signature $\mu\in\GT_N$, the function
$\wh\si_\mu\in C_0(\Om)$ belongs to the domain of the generator $A_\zw$ and
\begin{equation}\label{eq6.I}
A_\zw\wh\si_\mu=q(\mu,\mu)\wh\si_\mu+\sum_{\nu\in\GT_N:\,
\nu\ne\mu}r(\nu,\mu)\wh\si_\nu.
\end{equation}
\end{proposition}

\begin{proof}
For $\mu\in\GT_N$, let $\one_\mu$ denote the function on $\GT_N$ defined by
$\one_\mu(\nu)=\de_{\mu\nu}$. By the definition of $D_{\zw\mid N}$, see
\eqref{eq6.H},
\begin{equation}\label{eq6.L}
D_{\zw\mid N}\one_\mu=\sum_{\nu\in\GT_N}q(\nu,\mu)\one_\nu.
\end{equation}

For any $\la\in\GT_N$ we set
\begin{equation}\label{eq6.X}
\One_\la=(\Dim_N\la)^{-1}\,\one_\la.
\end{equation}
Then, by \eqref{eq6.D}, formula \eqref{eq6.L} can be rewritten as
\begin{equation}\label{eq6.M}
D_{\zw\mid N}\One_\mu=q(\mu,\mu)\One_\mu+\sum_{\nu\in\GT_N:\,\nu\ne\mu}
r(\nu,\mu)\One_\nu.
\end{equation}

Claim (iii) of Proposition \ref{prop6.B} implies that all finitely supported
functions on $\GT_N$ belong to the domain of $A_{\zw\mid N}$ and for every such
function $f$ one has $A_{\zw\mid N}f=D_{\zw\mid N}f$. In particular, taking
$f=\One_\mu$ we obtain from \eqref{eq6.M}
\begin{equation}\label{eq6.K}
A_{\zw\mid N}\One_\mu=q(\mu,\mu)\One_\mu+\sum_{\nu\in\GT_N:\,\nu\ne\mu}
r(\nu,\mu)\One_\nu.
\end{equation}

Next, by virtue of Proposition \ref{prop6.C} one can apply Proposition
\ref{prop5.C}, claim (i). It implies that for every $\la\in\GT_N$, the function
$\La^\infty_N\One_\la$ on $\Om$ belongs to the domain of the generator $A_\zw$
and
$$
A_\zw\La^\infty_N\One_\la=\La^\infty_N A_{\zw\mid N}\One_\la
$$
(recall that the links $\La^\infty_N:\Om\dasharrow \GT_N$ are defined in \eqref{eq5.A}).
Together with \eqref{eq6.K} this gives
\begin{equation}\label{eq6.N}
A_\zw\La^\infty_N\One_\mu=q(\mu,\mu)\La^\infty_N
\One_\mu+\sum_{\nu\in\GT_N:\,\nu\ne\mu} r(\nu,\mu)\La^\infty_N\One_\nu.
\end{equation}

Finally, \eqref{eq5.A} shows that for any $\la\in\GT_N$
$$
\La^\infty_N\One_\la=\wh\si_\la.
$$
Substituting this into \eqref{eq6.N} gives the desired formula.
\end{proof}

Let $\wh\F\subset C_0(\Om)$ denote the  linear span of the functions
$\wh\si_\la$, where $\la$ range over $\GT_1\sqcup
\GT_2\sqcup\GT_3\sqcup\dots$\,. As was shown in the proof of Proposition
\ref{prop6.E}, $\wh\F$ coincides with the linear span of the spaces
$\La^\infty_N C_c(\GT_N)$, where $N=1,2,3,\dots$ and $C_c(\GT_N)\subset
C_0(\GT_N)$ stands for the subspace of finitely supported functions. By
Proposition \ref{prop6.E}, $\wh\F$ is contained in the domain of the generator
$A_\zw$. Moreover, this proposition explains how the generator acts on $\wh\F$.
In particular, we see that $\wh\F$ is invariant under the action of the
generator.

\begin{theorem}\label{thm6.A}
The subspace  $\wh\F\subset C_0(\Om)$ is a core for the generator $A_\zw$. 
\end{theorem}
 
 This fact is not used in the arguments below, but it is a substantial complement to our main result, Theorem \ref{thm7.B}, which describes explicitly the operator  $A_\zw\big|_{\wh\F}$ (the restriction of the generator to $\wh\F$). By virtue of Theorem \ref{thm6.A}, the latter operator uniquely determines the generator, so Theorem \ref{thm7.B} contains, in principle, a complete information about the generator. 
 
 \begin{proof}
 Theorem \ref{thm6.A} is proved in \cite{Ols-FAA}. Here we only indicate the idea of the proof. By Proposition \ref{prop5.C},  it
suffices to show that $C_c(\GT_N)$ is a core for $A_{\zw\mid N}$ for every $N$.
This, in turn, can be verified as in Borodin-Olshanski \cite{BO-EJP}, by making
use of a result due to Ethier and Kurtz (see its formulation in \cite[Theorem
2.3 (iv)]{BO-EJP}. (Note two misprints in \cite{BO-EJP}: the claims of
Corollary 6.6 (ii) and Corollary 8.7 (ii) concern the subspace of finitely
supported functions, so that instead of $C_0(\,\cdot\,)$ one should read
$C_c(\,\cdot\,)$.)
\end{proof}

\section{The main theorem}\label{sect7}

\subsection{Formulation of the main theorem}

In Section \ref{sect4}, we defined the differential operator $\D_\zw$ which
acts on $R$. It depends on an arbitrary quadruple $(\zw)\in\C^4$. We also
showed that it preserves the ideal $J\subset R$ and so determines an operator
on the quotient $\wh R=R/J$. Let us denote the latter operator by $\wh\D_\zw$.

Given $\psi\in R$, we will denote by $\wh\psi\in\wh R$ the image of $\psi$
under the canonical map $R\to\wh R$. In particular, we may speak about the
elements $\wh\si_\la\in\wh R$. Note that in Section \ref{sect3}, we already used the same notation: namely, given $\psi\in\RR$, we denoted by $\wh\psi(\om)$ the corresponding function on $\Om$ (its definition is given just before Proposition \ref{prop3.A}). Formally, the two definitions of $\wh\psi$ look differently, but the new definition is morally an extension of the old one, because, as shown in Proposition \ref{prop3.A}, the kernel of the homomorphism $\RR\ni \psi\mapsto \wh\psi(\,\cdot\,)$ coincides with $J\cap\RR$.

\begin{theorem}[Main Theorem]\label{thm7.B}
Let $(\zw)\in\C^4$ be an
admissible quadruple of parameters $(\zw)\in\C^4$, see Definition \ref{def6.A}, and recall that $A_\zw$ denotes the generator of the semigroup $P_{\zw\mid\infty}(t)$, see Definition \ref{def6.B}.  We
restrict $A_\zw$  to the core $\wh\F\subset C_0(\Om)$ defined in the end of
Section \ref{sect6}. Finally, let $\la$ range over the set of all signatures,
except $\la=\varnothing$.

Under the identification of the elements $\wh\si_\la\in\wh R$ with the
functions $\wh\si_\la(\om)$ from the core $\wh\F$, the action of the generator
$A_\zw$ on those functions coincides with the action of the operator\/
$\wh\D_\zw$ on the corresponding elements $\wh\si_\la\in\wh R$.
\end{theorem}

\begin{remark}\label{remark7.A}

As mentioned in the introduction, Theorem \ref{thm7.B} gives a precise sense to
the informal statement (Theorem \ref{thm1.A}) that ``the generator $A_\zw$ is
implemented by the differential operator $\D_\zw$''. It is tempting to regard
Theorem \ref{thm7.B} as the indication that $X_\zw$ are diffusion processes,
and it would be very interesting to find out whether this is true. For
instance, is it true that the operators $A_\zw$ are diffusion generators as
defined in Ledoux \cite[Section 1.1]{Ledoux}.
\end{remark}

Theorem \ref{thm7.B} will be proved in a slightly stronger form (Theorem
\ref{thm7.A} below).

We are going to define a linear operator $R\to R$ that mimics the action of the
generator $A_\zw$ on $\wh\F$. In the next proposition we use the $I$-adic
topology in $R$, introduced in Subsection \ref{sect2.A}.

\begin{proposition}\label{prop7.A}
For every quadruple $(\zw)\in\C^4$ there exists a unique linear operator
$\AA_\zw:R\to R$, continuous in the $I$-adic topology, annihilating the unity
element $1\in R$, and such that for every $N=1,2,\dots$ and every
$\mu\in\GT_N$,
\begin{equation}\label{eq7.A}
 \AA_\zw\si_\mu=q_{\zw\mid N}(\mu,\mu)\si_\mu+\sum_{\nu\in\GT_N:\,
\nu\ne\mu}r_{\zw\mid N}(\nu,\mu)\si_\nu,
\end{equation}
where $q_{\zw\mid N}(\mu,\mu)$ and $r_{\zw\mid N}(\nu,\mu)$ is a more detailed
notation for the quantities $q(\mu,\mu)$ and $r(\nu,\mu)$ defined in the
beginning of Subsection \ref{sect6.A}.
\end{proposition}

It is worth emphasizing that here we drop the admissibility condition on the
parameters imposed in Section \ref{sect6}: the operator $\AA_\zw$ is considered
for any complex values of $(\zw)$. This is possible because the formulas
defining the quantities $q(\mu,\mu)$ and $r(\nu,\mu)$ make sense for arbitrary
$(\zw)\in\C^4$.

\begin{proof}
Together with the condition $\AA_\zw 1=0$, formula \eqref{eq7.A} determines
$\AA_\zw$ on the linear span of the basis elements $\si_\mu$ including
$\si_\varnothing=1$. The continuity of this operator immediately follows from
the fact that $\AA_\zw\si_\mu$ is a linear combination of $\si_\mu$ and
``neighboring'' basis vectors of the form $\si_{\mu\pm \epsi_i}$. The explicit
form of the coefficients is not important here.
\end{proof}

The next claim will be used in Section \ref{sect9}.

\begin{proposition}\label{prop7.C}
For any $(\zw)\in\C^4$, the operator $\AA_\zw$ preserves the ideal $J\subset
R$.
\end{proposition}

\begin{proof}
It suffices to prove that $\AA_\zw$ commutes with the operator of multiplication by $\varphi$. We are going to show that the latter claim is merely a rephrasing of the commutation relation \eqref{eq6.O}.

Indeed, for every $N=0,1,2,\dots$ we define a linear isomorphism $I_N$ between  the space $R_N$ and the space $\Fun(\GT_N)$ of functions on the discrete set $\GT_N$ by setting
$$
I_N: \sum_{\mu\in\GT_N}a_\mu \si_\mu \mapsto \sum_{\mu\in\GT_N}a_\mu\wt\one_\mu,
$$
where $a_\mu$ are arbitrary complex coefficients. By the very definition of $\AA_\zw$, we have
$$
\AA_\zw\big|_{R_N}=I_N^{-1}D_{\zw\mid N}I_N.
$$

On the other hand, 
Proposition \ref{prop3.AA} says that for every $\mu\in\GT_N$,
$$
\varphi\si_\mu=\sum_{\la:\, \la\succ\mu}\si_\la.
$$
Comparing this with the definition of the canonical link $\La^{N+1}_N$ (see \eqref{eq5.C}) and the definition of $\wt\one_\mu$ (see \eqref{eq6.X}) we conclude that the operator $R_N\to R_{N+1}$ given by multiplication by $\varphi$ coincides with the operator $I^{-1}_{N+1}\La^{N+1}_N I_N$.

Therefore, the commutation relation \eqref{eq6.O} just means that $\AA_\zw$ and multiplication by $\varphi$ commute.  
\end{proof}

\begin{theorem}\label{thm7.A}
Let $(\zw)$ be an arbitrary quadruple of complex parameters. The operator
$\AA_\zw:R\to R$ from Proposition \ref{prop7.A} coincides with the differential
operator\/ $\D_\zw$ introduced in Definition \ref{def4.A}.
\end{theorem}

The theorem says that for every signature $\mu$, the element
$\psi:=\D_\zw\si_\mu$ is a \emph{finite} linear combination of basis elements
$\si_\nu$ (which is not evident!) and the corresponding function $\wh\psi$
coincides with $A_\zw\wh\si_\mu$. Obviously, this implies Theorem \ref{thm7.B}.

The rest of the paper is devoted to the proof of Theorem \ref{thm7.A}. The main
essence of  difficulty is the fact that $\AA_\zw$ is defined by its action on
the elements of the basis $\{\si_\mu\}$, whereas the action of $\D_\zw$ is
directly seen in another basis, $\{\varphi_\nu\}$. The transition coefficients
between the two bases seem to be too complicated to allow a direct verification
of the theorem. 

In Subsection \ref{sect7.B} we outline the plan of the proof, but first we need to recall a necessary formalism.

\subsection{Abstract differential operators}

Let $\A$ be a commutative
unital algebra and $\mathscr D:\A\to\A$ be a linear operator. For $x\in\A$, let
$M_x:\A\to\A$ denote the operator of multiplication by $x$. Let us say that
$\mathscr D$ has order $\le k$ (where $k=0,1,2,\dots$) if its $(k+1)$-fold
commutator with operators of multiplication by arbitrary elements of the
algebra vanishes:
$$
[M_{x_1},[M_{x_2,},\dots [M_{x_{k+1}},\mathscr D]\dots]]=0, \qquad
x_1,\dots,x_{k+1}\in\A.
$$

Let $x_1,x_2,\dots$ be an arbitrary collection of elements of $\A$. If
$\mathscr D:\A\to\A$ has order $\le k$, then its action on all monomials of any
degree, formed from $\{x_i\}$, is uniquely determined provided one knows the
action on the monomials of degree $\le k$, including the monomial of degree
$0$, which is 1.

We give a proof for $k=2$ because we need this case only.

\begin{proposition}\label{prop7.B}
Let, as above, $\A$ be a commutative unital algebra and $\DD:\A\to\A$ be a
linear operator of order $\le2$. For any elements $x_1,\dots,x_n\in\A$, where
$n\ge3$, one has {\rm(}below the indices range over $1,\dots,n${\rm)}
\begin{multline}\label{eq7.B}
\DD(x_1\dots x_n)\\=\sum_{i<j}\left(\prod_{k:\, k\ne i,j}x_k\right)
\DD(x_ix_j)-\sum_{i}\left(\prod_{k:\, k\ne i}x_k\right) \DD x_i
+\left(\prod_{k}x_k\right)\DD 1.
\end{multline}
\end{proposition}

\begin{proof}
Assume first that $\DD$ has order $\le0$. This means $[\DD,M_x]=0$ for any
$x\in\A$. Then
\begin{equation}\label{eq7.C}
\DD x=\DD M_x 1=M_x\DD 1=x\DD 1.
\end{equation}

Next, assume $\DD$ has order $\le1$. This means that $[\DD,M_x]$ has order
$\le0$. Then, using \eqref{eq7.C}, we have for any $x,y\in\A$
\begin{equation}\label{eq7.D}
\DD(xy)=\DD M_x y=x \DD y+[\DD, M_x]y=x \DD y+y[\DD,M_x]1=x\DD y+y\DD
x-xy\mathscr D1.
\end{equation}

Finally, assume $\DD$ has order $\le2$. We are going to show that for any
$x,y,z\in\A$
\begin{equation}\label{eq7.E}
\DD(xyz)=x\DD(yz)+y\DD(xz)+z\DD(xy)-xy\DD z-xz\DD y-yz\DD x+xyz\DD 1.
\end{equation}
Once this is established, the desired formula \eqref{eq7.B} is verified by
induction on $n$. Namely, \eqref{eq7.E} is the base of the induction ($n=3$),
and in order to pass from $n$ to $n+1$ one applies \eqref{eq7.E} with
$x=x_1\dots x_{n-1}$, $y=x_n$, $z=x_{n+1}$.

It remains to prove \eqref{eq7.E}, which is achieved using the same trick. We
have
$$
\DD(xyz)=\DD M_x(yz)=x\DD(yz)+[\DD,M_x](yz).
$$
As $[\DD,M_x]$ has order $\le1$, we may apply \eqref{eq7.D}, which gives
$$
[\DD,M_x](yz)=y[\DD, M_x]z+z[\DD, M_x]y-yz[\DD,M_x]1.
$$
Next,
$$
y[\DD, M_x]z=y\DD(xz)-xy\DD z, \qquad z[\DD, M_x]y=z\DD(xy)-xz\DD y
$$
and
$$
-yz[\DD,M_x]1=-yz\DD x+xyz\DD 1.
$$
Putting all the pieces together we get \eqref{eq7.E}.
\end{proof}

\subsection{Plan of proof}\label{sect7.B}

The proof of Theorem \ref{thm7.A} is reduced to the following two claims.

\begin{claim}\label{claim1}
The operators\/ $\D_\zw$ and\/ $\AA_\zw$ coincide on the monomials of degree
$\le 2$.
\end{claim}

\begin{claim}\label{claim2}
The operator $\AA_\zw: R\to R$ has order $\le2$ in the abstract sense.
\end{claim}

\begin{proof}[Derivation of the theorem from these claims]
Since both operators are continuous in the $I$-adic topology of $R$, it
suffices to prove that they coincide on the monomials
$\varphi_\nu=\varphi_{\nu_1}\dots\varphi_{\nu_N}$.

Since $\D_\zw$ is a second order differential operator, it has order $\le2$ in
the abstract sense. The same holds for the operator $\AA_\zw$, by virtue of
Claim \ref{claim2}. Thus, both operators have order $\le2$.

Therefore, by Proposition \ref{prop7.B}, it suffices to know that the two
operators coincide on monomials of degree $N\le2$, and this holds by virtue of
Claim \ref{claim1}.
\end{proof}

Claims \ref{claim1} and \ref{claim2} are  proved in Section \ref{sect8} and
\ref{sect9}, respectively.

The structure of the proof reflects the way of how the differential operator
$\D_\zw$ has been found. Namely, assuming that $\AA_\zw$ is a second order
differential operator we may write down it explicitly by computing its action
on the monomials of degree $\le2$, and this what we actually do in the proof of
Claim \ref{claim1}.

The proof is indirect, but it seems to me that a direct verification of the
equality $\D_\zw=\AA_\zw$, without recourse to Claim \ref{claim2}, is a
difficult task.

\section{Proof of Claim 7.7}\label{sect8}

\subsection{Beginning of proof}

The differential operator $\D_\zw$ does not contain terms of order 0, so it
annihilates the constants. The same holds for the operator $\AA_\zw$, by the
very definition.

Let us verify that
$$
\D_\zw\varphi_n=\AA_\zw\varphi_n, \qquad n\in\Z.
$$

By the definition of $\D_\zw$, the left-hand side equals
$$
\begin{gathered}
B_n=(n+w+1)(n+w'+1)\varphi_{n+1}+(n-z-1)(n-z'-1)\varphi_{n-1}\\
-\bigl((n-z)(n-z')+(n+w)(n+w')\bigr)\varphi_n.
\end{gathered}
$$

To compute the right-hand side we observe that $\varphi_n=\si_{(n)}$ and then
use the definition of $\AA_\zw$ (see \eqref{eq7.A}). It says that
$$
\AA_\zw\varphi_n=q(n,n)\varphi_n+r(n+1,n)\varphi_{n+1}+r(n-1,n)\varphi_{n-1}.
$$
Here the quantities  $r(n\pm1,n)$ and $q(n,n)$ are given by formulas
\eqref{eq6.E}, \eqref{eq6.F}, and \eqref{eq6.G}, where we take $N=1$, so that
$n$ and $n\pm1$ denote signatures of length 1. We get first
$$
r(n,n+1)=(z-n)(z'-n), \quad r(n,n-1)=(w+n)(w'+n),
$$
which implies
$$
r(n-1,n)=(z-n+1)(z'-n+1), \quad r(n+1,n)=(w+n+1)(w'+n+1).
$$
Next,
$$
q(n,n)=-r(n,n+1)-r(n,n-1)=-(z+n)(z'+n)-(w+n)(w'+n).
$$
This gives the same quantity $B_n$, as desired.

A more difficult task is to check that the two operators coincide on quadratic
monomials. That is,
\begin{equation}\label{eq8.N}
\D_\zw\varphi_\ka=\AA_\zw\varphi_\ka, \qquad \ka=(k_1,k_2)\in\Z^2, \quad k_1\ge
k_2.
\end{equation}
The rest of the section is devoted to the proof of this equality.

Below we use the notation:
$$
\de:=\epsi_1-\epsi_2=(1,-1)\in\Z^2.
$$

\subsection{Step 1}

By \eqref{eq2.F},
$$
\varphi_\ka=\sum_{p=0}^\infty \si_{\ka+p\de}.
$$

So far we used the notation $r(\nu,\mu)$ for $\nu=\mu\pm \epsi_i$ only, but now
it will be convenient to write $r(\mu,\mu)$ instead of $q(\mu,\mu)$. With this
agreement we have
$$
\AA_\zw\varphi_\ka=\sum_{p=0}^\infty\sum_\epsi
r(\ka+p\de+\epsi,\ka+p\de)\si_{\ka+p\de+\epsi},
$$
where $\epsi$ ranges over $\{\pm \epsi_1, \pm \epsi_2, 0\}$.

Next, by \eqref{eq2.G},
$$
\si_{\ka+p\de+\epsi}=\varphi_{\ka+p\de+\epsi}-\varphi_{\ka+(p+1)\de+\epsi}.
$$
Consequently,
\begin{multline}\label{eq8.A}
\AA_\zw\varphi_\ka=\sum_{\epsi}r(\ka+\epsi,\ka)\varphi_{\ka+\epsi}
+\sum_{p=1}^\infty\sum_\epsi\big[r(\ka+p\de+\epsi,\ka+p\de)\\-r(\ka+(p-1)\de+\epsi,
\ka+(p-1)\de)\big]\varphi_{\ka+p\de+\epsi}.
\end{multline}

The right-hand side is a linear combination of elements $\varphi_{\,l_1l_2}$
such that the difference $(l_1+l_2)-(k_1+k_2)$ takes only three possible
values: $\pm1$ and $0$. According to this we write $\AA_\zw\varphi_\ka$ as the
sum of three components,
$$
\AA_\zw\varphi_\ka=(\AA_\zw\varphi_\ka)_1+(\AA_\zw\varphi_\ka)_{-1}+(\AA_\zw\varphi_\ka)_0.
$$

On the other hand, it follows from \eqref{eq4.B} and \eqref{eq4.C} that
$\D_\zw\varphi_\ka$ has the same property, so we write
$$
\D_\zw\varphi_\ka=(\D_\zw\varphi_\ka)_1+(\D_\zw\varphi_\ka)_{-1}+(\D_\zw\varphi_\ka)_0.
$$

Thus we are led to check three equalities,
\begin{equation}\label{eq8.B}
 \begin{gathered}(\AA_\zw\varphi_\ka)_1=(\D_\zw\varphi_\ka)_1, \qquad
(\AA_\zw\varphi_\ka)_{-1}=(\D_\zw\varphi_\ka)_{-1},\\
(\AA_\zw\varphi_\ka)_0=(\D_\zw\varphi_\ka)_0.
\end{gathered}
\end{equation}

The first two equalities are equivalent because of the symmetry consisting in
switching
$$
(z,z') \leftrightarrow(w,w'), \quad (k_1,k_2)\leftrightarrow (-k_2,-k_1),
\qquad (l_1,l_2)\leftrightarrow(-l_2,-l_1).
$$
Therefore, it suffices to check the first and third equalities in
\eqref{eq8.B}.

\subsection{Step 2}

On this step, we write down explicitly the component $(\AA_\zw\varphi_\ka)_1$
of \eqref{eq8.A}. It collects the contribution from the terms with
$\epsi=\epsi_1$ and $\epsi=\epsi_2$. Because $\de=\epsi_1-\epsi_2$, we have
$$
p\de+\epsi_1=(p+1)\de+\epsi_2.
$$
Using this relation one can write $(\AA_\zw\varphi_\ka)_1$ in the following
form:
\begin{equation}\label{eq8.G}
(\AA_\zw\varphi_\ka)_1=X_1+X_2,
\end{equation}
where
\begin{equation}\label{eq8.C}
\begin{gathered}
X_1:=r(\ka+\epsi_1,\ka)\varphi_{\ka+\epsi_1}+r(\ka+\epsi_2,\ka)\varphi_{\ka+\epsi_2}\\
+\big[r(\ka+\de+\epsi_2,\ka+\de)-r(\ka+\epsi_2,\ka)\big]\varphi_{\ka+\de+\epsi_2}
\end{gathered}
\end{equation}
and
\begin{equation}\label{eq8.D}
\begin{gathered}
X_2:=\sum_{p=1}^\infty\big[r(\ka+p\de+\epsi_1,\ka+p\de)-r(\ka+(p-1)\de+\epsi_1,\ka+(p-1)\de)\\
+r(\ka+(p+1)\de+\epsi_2,\ka+(p+1)\de)-r(\ka+p\de+\epsi_2,\ka+p\de)\big]\varphi_{\ka+p\de+\epsi_1}.
\end{gathered}
\end{equation}

To proceed further we need the explicit values of the jump rates: if
$\la=(l_1,l_2)$ with $l_1\ge l_2$, then
\begin{gather}
r(\la+\epsi_1,\la)=(w+l_1+2)(w'+l_1+2), \label{eq8.E}\\
r(\la+\epsi_2,\la)=\begin{cases}(w+l_2+1)(w'+l_2+1), &\text{if $l_1>l_2$}\\ 0,
& \text{if $l_1=l_2$}\end{cases}.\label{eq8.F}
\end{gather}

Let us substitute this in \eqref{eq8.D}. Then $\la=\ka+p\de$ or
$\la=\ka+(p+1)\de$ with $p\ge1$, and in both cases one has $l_1>l_2$. After a
simple computation one finds
\begin{equation*}
X_2=2\sum_{p=1}^\infty(2p+1+k_1-k_2)\varphi_{\ka+p\de+\epsi_1}.
\end{equation*}
It is convenient to extend the summation to $p=0$ and, to compensate this,
subtract from $X_1$ the term $2(k_1-k_2+1)\varphi_{\ka+\epsi_1}$.

Then we rewrite the decomposition \eqref{eq8.G} in a modified form:
\begin{equation}\label{eq8.H}
(\AA_\zw\varphi_\ka)_1=X'_1+X'_2,
\end{equation}
where
\begin{equation}\label{eq8.I}
X'_2=2\sum_{p=0}^\infty(2p+1+k_1-k_2)\varphi_{\ka+p\de+\epsi_1}
=2\sum_{p=0}^\infty(2p+1+k_1-k_2)\varphi_{k_1+p+1}\varphi_{k_2-m}
\end{equation}
and
\begin{equation*}
X'_1=X_1-2(k_1-k_2+1)\varphi_{\ka+\epsi_1}.
\end{equation*}

Finally, using again \eqref{eq8.E} and \eqref{eq8.F} one can check that
\begin{equation}\label{eq8.J}
X'_1=(w+k_1+1)(w'+k_1+1)\varphi_{k_1+1}\varphi_{k_2}
+(w+k_2+1)(w'+k_2+1)\varphi_{k_1}\varphi_{k_2+1}
\end{equation}

\subsection{Step 3}

Now let us turn to $(\D_\zw\varphi_\ka)_1$. This quantity results from an
appropriate truncation of the operator $\D$. Namely, we replace it by
$$
\D^{(1)}_\zw:=\sum_{n\in\Z}A^{(1)}_{nn}\frac{\pd^2}{\pd\varphi_n^2}+2\sum_{\substack{n_1,n_2\in\Z\\
n_1>n_2}} A^{(1)}_{n_1 n_2}\frac{\pd^2}{\pd\varphi_{n_1}\pd\varphi_{n_2}}
+\sum_{n\in\Z}B^{(1)}_n\frac{\pd}{\pd\varphi_n}, \
$$
where, for any indices $n_1\ge n_2$,
\begin{equation*}
A^{(1)}_{n_1
n_2}=\sum_{p=0}^\infty(n_1-n_2+2p+1)\varphi_{n_1+p+1}\varphi_{n_2-p}
\end{equation*}
and, for any $n\in\Z$,
\begin{equation*}
B^{(1)}_n=(n+w+1)(n+w'+1)\varphi_{n+1}.
\end{equation*}

We represent $(\D_\zw\varphi_\ka)_1=\D^{(1)}_\zw\varphi_\ka$ as the sum of two
components, the one coming from the action of the first order derivatives and
the other coming from the action of the second order derivatives. {}From the
explicit expressions above one can readily check that these two components
coincide with $X'_1$ and $X'_2$, respectively.

This completes the proof of the identity
$(\AA_\zw\varphi_\ka)_1=(\D_\zw\varphi_\ka)_1$, which is the first equality in
\eqref{eq8.B}. Now we apply similar arguments to prove the third equality in
\eqref{eq8.B}.

\subsection{Step 4 (cf. Step 2 above)}
Here we compute $(\AA_\zw\varphi_\ka)_0$. {}From \eqref{eq8.A} we obtain
\begin{equation}\label{eq8.K}
\begin{aligned}
(\AA_\zw\varphi_\ka)_0=r(\ka,\ka)\varphi_\ka &+\sum_{p=1}^\infty\big[r(\ka+p\de,\ka+p\de)\\
&-r(\ka+(p-1)\de, \ka+(p-1)\de)\big]\varphi_{\ka+p\de}.
\end{aligned}
\end{equation}

Recall that $r(\la,\la):=q(\la,\la)$. By \eqref{eq6.G}, for $\la=(l_1,l_2)$
with $l_1\ge l_2$,
\begin{equation*}
\begin{aligned}
r(\la,\la)=&-(z-l_1)(z'-l_1)-(w+l_1+1)(w'+l_1+1)\\
&-(z-l_2+1)(z'-l_2+1)-(w+l_2)(w'+l_2)\\
&+z+z'+w+w'+2.
\end{aligned}
\end{equation*}

We substitute this into \eqref{eq8.K} and obtain
$$
(\AA_\zw\varphi_\ka)_0=Y_1+Y_2,
$$
where
\begin{equation}\label{eq8.L}
\begin{aligned}
Y_1=\big\{&-(z-l_1)(z'-l_1)-(w+l_1+1)(w'+l_1+1)\\
&-(z-l_2+1)(z'-l_2+1)-(w+l_2)(w'+l_2)\\
&+z+z'+w+w'+2\big\}\varphi_\ka
\end{aligned}
\end{equation}
and
\begin{equation}\label{eq8.M}
Y_2=-2\sum_{p=1}^\infty(k_1-k_2+p)\varphi_{\ka+p\de}.
\end{equation}

\subsection{Step 5 (cf. Step 3 above)}

Let us turn to $(\D_\zw\varphi)_0$. We write
$$
(\D_\zw\varphi)_0=\D^{(0)}_\zw\varphi_\ka
$$ 
with $\D^{(0)}_\zw$ being the following
truncated operator:
$$
\D^{(0)}_\zw=\sum_{n\in\Z}A^{(0)}_{nn}\frac{\pd^2}{\pd\varphi_n^2}+2\sum_{\substack{n_1,n_2\in\Z\\
n_1>n_2}} A^{(0)}_{n_1 n_2}\frac{\pd^2}{\pd\varphi_{n_1}\pd\varphi_{n_2}}
+\sum_{n\in\Z}B^{(0)}_n\frac{\pd}{\pd\varphi_n}, \
$$
where, for any indices $n_1\ge n_2$,
\begin{equation*}
A^{(0)}_{n_1 n_2}= -(n_1-n_2)\varphi_{n_1}\varphi_{n_2}
\end{equation*}
and, for any $n\in\Z$,
\begin{equation*}
B^{(0)}_n=-\bigl((n-z)(n-z')+(n+w)(n+w')\bigr)\varphi_n.
\end{equation*}

It is readily seen that the result of the action on $\varphi_\ka$ of the first
derivatives in $\D^{(0)}_\zw$ coincides with $Y_1$ (see \eqref{eq8.L}), while the
action of the second derivatives leads to $Y_2$ (see \eqref{eq8.M}).

This completes the proof of \eqref{eq8.N}. Thus, Claim \ref{claim1} is proved,
too.

\section{Proof of Claim 7.8}\label{sect9}

\subsection{Reduction of the problem}

Let us fix two nonnegative integers $k$ and $l$, not equal both to $0$.

\begin{proposition}[cf. Proposition \ref{prop4.C}]\label{prop9.D}
If $z=k$ and $w=l$ as above, then the operator $\AA_\zw$ preserves the ideal
$I(k,-l)$, the kernel of the canonical map $R\to R(k,-l)$.
\end{proposition}

\begin{proof}
Let us set $\GT(k,-l)=\bigcup_{N=1}^\infty\GT_N(k,-l)$ (recall that the
definition of $\GT_N(n_+,n_-)$ is given in \eqref{eq5.B}). The ideal $I(k,-l)$
is the closed linear span of the basis elements $\si_\nu$ such that
$\nu\notin\GT(k,-l)$, where the closure is taken in the $I$-adic topology.
Therefore, it suffices to prove the following: if $\nu\notin\GT_N(k,-l)$ and
$\mu\in\GT_N(k,-l)$, then the quantity $r_{k,z',l,w'\mid N}(\nu,\mu)$ vanishes.

Next, this claim is readily verified by using the definition  of $r_{\zw\mid
N}(\nu,\mu)$, see \eqref{eq6.E} and \eqref{eq6.F}.
\end{proof}

By Proposition \ref{prop5.C}, $\AA_\zw$ preserves the ideal $J$ (for arbitrary
$(\zw)$). Therefore, if $z=k$ and $w=l$, then $\AA_\zw$ preserves the ideal
$J(k,-l)$ generated by $J$ and $I(k,-l)$, and hence gives rise to an operator
on the quotient algebra
$$
\wh
R(k,-l)=R/J(k,-l)=\C[\varphi_{-l},\dots,\varphi_k]\big/(\varphi_{-l}+\dots+\varphi_k-1),
$$
(this quotient has already appeared in \eqref{eq3.C}). Let us denote the latter
operator by $\bar\AA_{k,z',l,w'}$.

\begin{proposition}\label{prop9.E}
To prove Claim \ref{claim2} it suffices to show that the operators
$\bar\AA_{k,z',l,w'}$ have order $\le2$.
\end{proposition}

\begin{proof}
Indeed,  Proposition \ref{prop7.B} says that Claim \ref{claim2}  is equivalent to
the relation
$$
[M_{\psi_3},[M_{\psi_2},[M_{\psi_1},\AA_\zw]]]\psi_4=0,
$$
which has to hold for arbitrary four elements $\psi_1,\psi_2,\psi_3, \psi_4\in
R$. Without loss of generality we may assume that all these elements are
homogeneous. Then the left-hand side is homogeneous, too, as it follows from
the definition of $\AA_\zw$. By virtue of Proposition \ref{prop3.D}, it suffices
to prove that the left-hand side belongs to $J$. Because  $\AA_\zw$ preserves
$J$ (Proposition \ref{prop7.C}), this allows us to pass from from $R$ to its
quotient $\wh R=R/J$.

Next, we want to specify $z=k$, $w=l$ and to reduce the desired relation modulo
the ideal $I(k,-l)$. This is possible for the following reasons:

1.  $\AA_\zw$ depends quadratically on the parameters, which allows us to
specialize $(\zw)$ to any Zariski dense subset of $\C^4$ (or even any subset
which is a set of uniqueness for quadratic polynomials);

2.  as $k,l\to +\infty$, the ideals $I(k,-l)$ decrease and their intersection
is $\{0\}$;

3.  we know that the operator $A_\zw$ can be reduced modulo $I(k,-l)$ provided
that $z=k$ and $w=l$.
\end{proof}

As the result of the factorization modulo both $J$ and $I(k,-l)$ the algebra
$R$ is reduced to the algebra
$$
\wh
R(k,-l):=\C[\varphi_{-l},\dots,\varphi_k]\big/(\varphi_{-l}+\dots+\varphi_k-1),
$$
which is isomorphic to the algebra of polynomials in $m:=k+l$ variables (we
have $m+1$ variables subject to a linear relation). This substantially
simplifies our task, because instead of the operators $\AA_\zw$ acting on the
huge space $R$ we may deal with the operators $\bar\AA_{k,z',l,w'}$ acting on
algebras of polynomials.

We have a large freedom in the choice of parameters $(z',w')$, because the
argument above allows us to restrict them to an arbitrary set which is a set of
uniqueness for quadratic polynomials. For the reasons that will become clear
below it is convenient to set $z'=k+a$, $w'=l+b$, where $a$ and $b$ are real
numbers $>-1$.

Thus, we have to show that the operator $\bar\AA_{k,k+a,l,l+b}$, which acts on
the algebra $\wh R(k,-l)$, is of order $\le2$.

As explained in Subsection \ref{sect3.B}, we may realize $\wh R(k,-l)$ as the
algebra of polynomial functions on the simplex $\Om(k,-l)$. Our aim is to show
that in this realization, $\bar\AA_{k,k+a,l,l+b}$ is given by a second order
partial differential operator (the Jacobi operator). This will evidently imply
that it has order $\le2$ in the abstract sense.

Finally, it is readily seen that the operator $\AA_\zw$ behaves exactly as
$\D_\zw$ with respect to the shift of variables
$\varphi_n\mapsto\varphi_{n+\const}$ (see Proposition \ref{prop4.C}). This
allows us to assume, without loss of generality, that $l=0$, which slightly
simplifies the notation.

Thus, in what follows we assume that
\begin{equation}\label{eq9.M}
z=m, \quad z'=m+a, \quad w=0, \quad w'=b,
\end{equation}
where $m=1,2,\dots$ and $a,b>-1$, and we are dealing with the operator
$\bar\AA_{m,m+a,0,b}$ acting on $\wh R(m,0)$.

\subsection{The Jacobi differential operators}\label{sect9.B}

As in Subsection \ref{sect3.B} we introduce new variables $t_1,\dots,t_m$
related to $\varphi_0,\dots,\varphi_m$ in the following way:
$$
\sum_{n=0}^m\varphi_nu^n=\prod_{i=1}^m (t_i+(1-t_i)u),
$$
where $u$ is a formal variable. In other words, we substitute for
$\varphi_0,\dots,\varphi_m$ certain symmetric polynomials in $t_1,\dots,t_m$.
Then we may identify $\wh R(m,0)$ with the algebra of symmetric polynomials in
variables $t_1,\dots,t_m$ (see Proposition \ref{prop3.B}). We also regard
$(t_1,\dots,t_m)$ as coordinates on $\Om(m,0)$ with the understanding that
$$
1\ge t_1\ge\dots\ge t_m\ge0.
$$

Let us introduce the \emph{Jacobi differential operator} on $[0,1]$:
$$
D^{(a,b)}=t(1-t)\frac{d^2}{dt^2}+[b+1-(a+b+2)t]\frac{d}{dt}.
$$
Its connection with the classic Jacobi orthogonal polynomials is explained
below (Subsection \ref{sect9.A}). Let us observe that
\begin{equation}\label{eq9.I}
D^\ba t^n=-n(n+a+b+1) t^n+\,\text{lower degree terms}, \qquad n=0,1,2,\dots\,.
\end{equation}

Let
$$
V_m=V_m(t_1,\dots,t_m):=\prod_{1\le i<j\le m}(t_i-t_j), \qquad m=1,2,\dots,
$$
and let
$$
D^\ba_{\text{\rm variable $t_i$}}:=t_i(1-t_i)\frac{\partial^2}{\partial
t_i^2}+[b+1-(a+b+2)t_i]\frac{\partial}{\partial t_i}
$$
be a copy of the Jacobi operator applied to the $i$th variable, $i=1,\dots,m$.
{}From  \eqref{eq9.I} and the fact that $V_m$ is the Vandermonde determinant it
follows that
\begin{equation}\label{eq9.E}
\left(\sum_{i=1}^m D^\ba_{\text{\rm variable $t_i$}}\right)V_m=-\const_{a,b,m}
V_m,
\end{equation}
where
\begin{equation}\label{eq9.F}
\const_{a,b,m}:=\sum_{n=0}^{m-1}n(n+a+b+1).
\end{equation}

Now we introduce the \emph{$m$-variate Jacobi differential operator},
$m=2,3,\dots$, by
\begin{gather}
D^{(a,b)}_m :=\frac1{V_m}\circ\left(\sum_{i=1}^m D^\ba_{\text{\rm variable
$t_i$}}\right)\circ
V_m+\const_{a,b,m}\label{eq9.G}\\
=\sum_{i=1}^m \left(t_i(1-t_i)\frac{\partial^2}{\partial
t_i^2}+\left[b+1-(a+b+2)t_i+\sum_{j:\, j\ne
i}\frac{2t_i(1-t_i)}{t_i-t_j}\right]\frac{\partial}{\partial
t_i}\right).\label{eq9.H}
\end{gather}
The meaning of \eqref{eq9.G} is that the partial differential operator $\sum_{i=1}^m D^\ba_{\text{\rm variable
$t_i$}}$ is conjugated by the operator of multiplication by the Vandermonde
$V_m$, and adding $\const_{a,b,m}$ kills the constant term that arises after
conjugation. The equality between \eqref{eq9.G} and \eqref{eq9.H} is verified
directly (actually, in what follows,  we use only \eqref{eq9.G}).

Note that, although the coefficients of the first order derivatives in
\eqref{eq9.H} have singularities along the diagonals $t_i=t_j$, the action of
$D^{(a,b)}_m$ on the space of symmetric polynomials is well defined. Indeed,
let us look at \eqref{eq9.G}: the operator of multiplication by $V_m$
transforms symmetric polynomials into antisymmetric ones, then the application
of the symmetric partial differential operator $\sum_{i=1}^m D^\ba_{\text{\rm variable
$t_i$}}$ leaves the
space of antisymmetric polynomials invariant, and finally division by $V_m$
transforms it back into the space of symmetric polynomials.

(The construction of a partial differential (or difference) operator related to multivariate orthogonal polynomials that
 we used in \eqref{eq9.G} (and also in \eqref{eq9.Q} below) is well known. The probabilistic meaning of this construction is
related to Doob's $h$-transform, see K\"onig \cite{Konig}.)

The arguments of the preceding subsection reduce Claim \ref{claim2} to the
following theorem.

\begin{theorem}\label{thm9.A}
As explained above, we identify $\wh R(m,0)$ with the algebra of symmetric
polynomials in $m$ variables $t_1,\dots,t_m$. Then the action of the operator
$\bar\AA_{m,m+a,0,b}$ on this algebra is implemented by the $m$-variate Jacobi
differential operator $D^{(a,b)}_m$.
\end{theorem}

The proof occupies the rest of the section. Here is the scheme of proof.

As explained in Subsection \ref{sect5.A}, we dispose of finite stochastic
matrices $\La^{N+1}_N: \GT_{N+1}(m,0)\dasharrow \GT_N(m,0)$ and the links
$\La^\infty_N:\Om(m,0)\dasharrow\GT_N(m,0)$. Let, as above, $C(\GT_N(m,0))$
stand for the space of functions on the finite set $\GT_N(m,0)$. The link
$\La^\infty_N$ maps $C(\GT_N(m,0))$ into $C(\Om(m,0))$, and the image is
actually contained in $\wh R(m,0)\subset C(\Om(m,0))$. As $N$ grows, this image
enlarges (because of the relation $\La^\infty_N=\La^\infty_{N+1}\La^{N+1}_N$)
and in the limit as $N\to\infty$ it exhausts the whole space $\wh R(m,0)$. This
point will be explained in more detail below.

Recall that the operator $\AA_\zw$ was defined through the difference operators
$D_{\zw\mid N}$. In the special case when $z=m$ and $w=0$, the $N$th difference
operator is well defined on the subset $\GT_N(m,0)$. {}From the definition of
operator $\bar\AA_{m,m+a,0,b}$ it follows that it is characterized by the
commutation relations
$$
\bar\AA_{m,m+a,0,b}\La^\infty_N=\La^\infty_N D_{m, m+a,0,b\mid N},
$$
where $N=1,2,\dots$ and the both sides are viewed as operators from the
finite-dimensional space $C(\GT_N(m,0))$ to $\wh R(m,0)$. We will prove that in
these relations, $\bar\AA_{m,m+a,0,b}$ can be replaced by the Jacobi operator
$D^{(a,b)}_m$. That is, one has
\begin{equation}\label{eq9.L}
D^{(a,b)}_m\La^\infty_N=\La^\infty_N D_{m, m+a,0,b\mid N},
\end{equation}
This will imply the desired equality $\bar\AA_{m,m+a,0,b}=D^{(a,b)}_m$.

The signatures  $\la\in\GT_N(m,0)$ can be viewed as Young diagrams contained in
the rectangular diagram
$$
(m^N):=(\,\underbrace{m,\dots,m}_N\,).
$$
Given such a diagram $\la$, we associate with it the complementary diagram
$\ka\subseteq(N^m)$: it is obtained from the shape $(m^N)\setminus\la$ by
rotation and conjugation.

The proof \eqref{eq9.L} is divided into three steps:

\smallskip

\emph{Step} 1. We express $\La^\infty_N$ in terms of $(t_1,\dots,t_m)$ and
$\ka$ (Proposition \ref{prop9.B}).

\smallskip

\emph{Step} 2. We show that under the correspondence $\la\leftrightarrow\ka$,
the difference operator $D_{m, m+a,0,b\mid N}$ in the right-hand side of
\eqref{eq9.L} turns into the \emph{$m$-variate Hahn difference operator}
(Proposition \ref{prop9.A}). As the result, \eqref{eq9.L} takes the form
\begin{equation}\label{eq9.D}
D^{(a,b)}_m\La^\infty_N=\La^\infty_N\De^{(a,b,N+m-1)}_m, \qquad N=1,2,\dots\,,
\end{equation}
where $\De^{(a,b,N+m-1)}_m$ is the Hahn difference operator in question.

\smallskip

\emph{Step} 3. We prove that $\La^\infty_N$ transforms the $m$-variate
symmetric Hahn polynomials into the respective $m$-variate symmetric Jacobi
polynomials (Proposition \ref{prop9.C}). Then the proof is readily completed.

\smallskip

We proceed to the detailed proof of the theorem.

\subsection{Step 1: transformation of the link  $\La^\infty_N$}

Let $\la$ range over the set of Young diagrams contained in the rectangle
$(m^N)$, and $\ka\subseteq(N^m)$ be the complementary diagram to $\la$. In more
detail,
$$
\ka=(N-\la'_m, \dots,N-\la'_1),
$$
where the diagram $\la'$ is conjugate to the diagram $\la$. Next, we set
$$
l_i:=\la_i+N-i, \quad i=1,\dots,N; \qquad  k_j=\ka_j+m-j, \qquad j=1,\dots,m.
$$
Evidently, $l_1>\dots>l_N$ and $k_1>\dots>k_m$.

\begin{lemma}\label{lemma9.C}
The set $\{0,\dots,N+m-1\}$ is the disjoint union of the sets
$\L:=\{l_1,\dots,l_N\}$ and $\K:=\{k_1,\dots,k_m\}$.
\end{lemma}

\begin{proof}
This is a well-known fact, see e.g. Macdonald \cite[ch. I, (1.7)]{Ma95}.
\end{proof}

Introduce a notation:
$$
M:=N+m-1, \quad \I=\{0,\dots,M\}.
$$
Next, for a finite collection of numbers $X=\{x_1>\dots>x_n\}$ we set
$$
V(X)=V_n(x_1,\dots,x_n)=\prod_{1\le i<j\le n}(x_i-x_j).
$$

\begin{lemma}\label{lemma9.A}
One has
\begin{equation}\label{eq9.A}
V(l_1,\dots,l_N)=\frac{0!1!\dots M!\,V(k_1,\dots,k_m)}{\prod\limits_{j=1}^m
k_j!(M-k_j)!}
\end{equation}
\end{lemma}

\begin{proof}
By the preceding lemma, $\I=\L\sqcup\K$, whence
\begin{equation}\label{eq9.P}
V(\I)=V(\K\sqcup \L)=V(\K)\cdot V(\L)\cdot \prod_{x\in \K}\prod_{y\in \L}|x-y|.
\end{equation}

For $x\in\K$, set
$$
f(x):=\prod_{z\in\I\setminus \{x\}}|x-z|
$$
and observe that
$$
\prod_{x\in \K}\prod_{y\in
\L}|x-y|=\frac{\prod\limits_{x\in\K}f(x)}{(V(\K))^2}.
$$
Substituting this into \eqref{eq9.P} gives
$$
V(\L)=\frac{V(\I)V(\K)}{\prod\limits_{x\in \K}f(x)}.
$$

On the other hand, it is readily checked that
$$
f(x)=x!(M-x)!
$$
and $V(\I)=0!1!\dots M!$. This completes the proof.
\end{proof}

\begin{proposition}\label{prop9.B}
Let $\om=\om(t_1,\dots,t_m)$ be the point of the simplex $\Om(m,0)$ with
coordinates $(t_1,\dots,t_m)$. In the notation introduced above,
$$
\La^\infty_N(\om;\la)= \const_{m,M}\frac{V(k_1,\dots,k_m)}{V(t_1,\dots,t_m)}
\det\left[\binom{M}{k_j}t_i^{k_j}(1-t_i)^{M-k_j}\right]_{i,j=1}^m,
$$
where
$$
\const_{m,M}=\prod_{i=1}^m\frac{(M-i+1)!}{M!}.
$$
\end{proposition}

In particular, in the simplest case $m=1$, there is a single coordinate
$t=t_1\in[0,1]$, the diagram $\la$ has a single column, the complementary
diagram has a single row whose length equals $\ka_1=k\in\{0,\dots,N\}$, and
$\La^\infty_N$ is represented as the link $[0,1]\dasharrow\{0,\dots,N\}$ that
assigns to a point $t\in[0,1]$ the binomial distribution on $\{0,\dots,N\}$
with parameter $t$.

\begin{proof}
(i) By the very definition of the link $\La^\infty_N$ (see \eqref{eq5.A} and
the comment after it),
$$
\La^\infty_N(\om,\la)=\Dim_N\la\cdot\left\{\text{\rm coefficient of
$s_\la(u_1,\dots,u_N)$ in $\Phi(u_1;\om)\dots\Phi(u_N;\om)$}\right\}.
$$

We have
\begin{gather*}
\Phi(u_1;\om)\dots\Phi(u_N;\om)=\prod_{i=1}^m\prod_{j=1}^N(1+\be^+_i(u_j-1))
=\prod_{i=1}^m\prod_{j=1}^N(t_i+(1-t_i)u_j)\\
=\prod_{i=1}^m(1-t_i)^N \cdot
\prod_{i=1}^m\prod_{j=1}^N\left(\frac{t_i}{1-t_i}+u_j\right)\\
=\prod_{i=1}^m(1-t_i)^N \cdot \sum_{\la: \la\subseteq
(m^N)}s_\ka\left(\frac{t_1}{1-t_1},\dots,\frac{t_m}{1-t_m}\right)s_\la(u_1,\dots,u_N),
\end{gather*}
where the last equality follows from the dual Cauchy identity, see
\cite[Chapter I, Section 4, Example 5]{Ma95}. Therefore,
$$
\La^\infty_N(\om,\la)=\Dim_N\la\cdot \prod_{i=1}^m(1-t_i)^N \cdot
s_\ka\left(\frac{t_1}{1-t_1},\dots,\frac{t_m}{1-t_m}\right).
$$

(ii) By the definition of the Schur polynomials,
$$
\prod_{i=1}^m(1-t_i)^N\cdot
s_\ka\left(\frac{t_1}{1-t_1},\dots,\frac{t_m}{1-t_m}\right)
=\frac{\prod\limits_{i=1}^m(1-t_i)^N\cdot
\det\left[\left(\dfrac{t_i}{1-t_i}\right)^{k_j}\right]}
{V\left(\dfrac{t_1}{1-t_1},\dots,\dfrac{t_m}{1-t_m}\right)},
$$
where the determinant in the numerator is of order $m$.

The denominator of this expression is equal to
$$
\prod_{i=1}^m(1-t_i)^{-m+1}\cdot V(t_1,\dots,t_m).
$$
Therefore, the whole expression is
$$
\frac{\prod\limits_{i=1}^m(1-t_i)^M\cdot
\det\left[\left(\dfrac{t_i}{1-t_i}\right)^{k_j}\right]} {V(t_1,\dots,t_m)} =
\frac{\det\left[t_i^{k_j}(1-t_i)^{M-k_j}\right]} {V(t_1,\dots,t_m)},
$$
so that
$$
\La^\infty_N(\om,\la)=\frac{\Dim_N\la}{V(t_1,\dots,t_m)}
\det\left[t_i^{k_j}(1-t_i)^{M-k_j}\right].
$$

(iii) It remains to handle $\Dim_N\la$. By Weyl's dimension formula,
$$
\Dim_N\la=\frac{V(\L)}{V(N-1,N-2,\dots,0)}=\frac{V(l_1,\dots,l_N)}{0!1!\dots(N-1)!}
$$

The numerator has been computed in Lemma \ref{lemma9.A}. Applying it we get
$$
\La^\infty_N(\om,\la)=\frac{0!1!\dots M!}{0!1!\dots
(N-1)!}\frac{V(k_1,\dots,k_m)}{V(t_1,\dots,t_m)}
\det\left[\frac1{k_j!(M-k_j)!}\,t_i^{k_j}(1-t_i)^{M-k_j}\right].
$$

The constant factor in front equals $\prod_{j=1}^m(M-j+1)!$. Dividing it by
$(M!)^m$ and introducing the same quantity inside the determinant we finally
get the desired expression.
\end{proof}

\subsection{Step 2: transformation of the difference operator $D_{m, m+a,0,b}$}

We continue to deal with two mutually complementary point configurations
$\L=(l_1>\dots>l_N)$ and $\K=(k_1>\dots>k_m)$ on the lattice interval
$\I=\{0,\dots,M\}$. Our next aim is to derive a convenient expression for the
jump rates introduced in Subsection \ref{sect6.A}. So far they were denoted as
$q(\nu,\nu\pm \epsi_i)$. Now we rename $\nu$ to $\la$ and next we pass from $\la$
to the corresponding point configuration $\L$. In terms of $\L$, the transition
$\la\to\la\pm \epsi_i$ can be written as $x\to x\pm1$, where $x=l_i$. According
to this we change the former notation for the jump rates and will denote them
by $q(x\to x\pm1)$, with the understanding that $x\in\L$.

Taking into account the values of the parameters (see \eqref{eq9.M}), the
formulas of Subsection \ref{sect6.A} can be rewritten as follows
\begin{gather}
q(x\to x+1)=\frac{V(\L-\{x\}+\{x+1\})}{V(\L)}\,(M-x)(M+a-x), \label{eq9.B}\\
q(x\to x-1)=\frac{V(\L-\{x\}+\{x-1\})}{V(\L)}\,x(b+x).\label{eq9.C}
\end{gather}
Here $\L-\{x\}+\{x\pm1\}$ denotes the configuration obtained from $\L$ by
removing $x$ and inserting $x\pm1$ instead.

Note that the transition $x\to x+1$ is forbidden if  the corresponding vector
$\la+\epsi_i$ is not a signature, which happens when $\la_{i-1}=\la_i$. In
terms of $\L$, this means $x+1\in\L$,  in which case the configuration
$\L-\{x\}+\{x+1\}$ contains the point $x+1$ twice, and then
$V(\L-\{x\}+\{x+1\})$ should be understood as $0$. Likewise, if $x\to x-1$ is
forbidden, then $V(\L-\{x\}+\{x-1\})$ vanishes. Thus, \eqref{eq9.B} and
\eqref{eq9.C} formally assign rate 0 to forbidden transitions, which is
reasonable.

\begin{lemma}\label{lemma9.D}
In terms of the complementary configuration $\K$, the jump rates take the form
\begin{gather*}
\wt q(y\to y-1)=\frac{V(\K-\{y\}+\{y-1\})}{V(\K)}y(M+1+a-y),\\
\wt q(y\to y+1)=\frac{V(\K-\{y\}+\{y+1\})}{V(\K)}(M-y)(b+y+1).
\end{gather*}
\end{lemma}

\begin{proof}
A jump $x\to x+1$ in $\L$ is possible if and only if  $x\in\L$ and
$x+1\notin\L$. This is equivalent to saying that $x+1\in\K$ and $x\notin \K$,
which in turn means the possibility of the jump $y\to y-1$, where $y=x+1$.
Therefore, $\wt q(y\to y-1)=q(x\to x+1)$.

Now we have to express the quantity $q(x\to x+1)$ given by \eqref{eq9.B} in
terms of $\K$. Lemma \ref{lemma9.A} tell us that
$$
V(\L)=\const\,\frac{V(\K)}{\prod\limits_{y\in\K}y!(M-y)!}.
$$
It follows that
$$
\frac{V(\L-\{x\}+\{x+1\})}{V(\L)}=\frac{V(\K-\{y\}+\{y-1\})}{V(\K)}\frac{y}{M+1-y}
$$
Next,
$$
(M-x)(M+a-x)=(M+1-y)(M+1+a-y).
$$
Multiplying out these two quantities we get the desired expression for $\wt
q(y\to y-1)$.

Likewise, the jump $x\to x-1$ is equivalent to $y\to y+1$, where $y=x-1$, so we
rewrite the expression for $q(x\to x-1)$ given by \eqref{eq9.C}. We have
$$
\frac{V(\L-\{x\}+\{x-1\})}{V(\L)}=\frac{V(\K-\{y\}+\{y+1\})}{V(\K)}\frac{M-y}{y+1}.
$$
Next,
$$
x(b+x)=(y+1)(b+y+1).
$$
Multiplying out these two quantities we get the desired expression for $\wt
q(y\to y+1)$.
\end{proof}

We introduce the \emph{Hahn difference operator} $\De^\baM$ by
\begin{equation}\label{eq9.J}
\begin{aligned}
(\De^\baM F)(y) =(y+b+1)(M-y)&[F(y+1)-F(y)]\\ +y(M+a-y+1)&[F(y-1)-F(y)],
\end{aligned}
\end{equation}
where $F$ is a function in variable $y$. Note that $\De^\baM$ is well defined
on $\I$. Indeed, the coefficient in front of $[F(y+1)-F(y)]$ vanishes at the
point $y=M$, the right end of the interval; likewise, the coefficient in front
of $[F(y-1)-F(y)]$ vanishes at the left end $y=0$.

The difference operator $\De^\baM$ is associated with the classic Hahn
polynomials: see Koekoek-Swarttouw \cite[(1.5.5)]{KS} and the next subsection.
Note that our parameters $(a,b,M)$ correspond to parameters $(\be,\al,N)$ from
\cite[Section 1.5]{KS}.

It is directly verified that
$$
\De^\baM y^n=-n(n+a+b+1)y^n+\text{lower degree terms}, \qquad n=0,1,2,\dots\,.
$$
Note that the factor in front of $y^n$  is exactly the same as in
\eqref{eq9.I}. In particular, it does not depend on the additional parameter
$M$ that enters the definition of the difference operator.

Now we introduce the \emph{$m$-variate Hahn difference operator} in the same
way as we defined above the $m$-variate Jacobi operator:
\begin{equation}\label{eq9.Q}
\De^\baM_m=\frac1{V_m}\circ\left(\sum_{i=1}^m\De^\baM_{\text{\rm variable\,}
y_i}\right)\circ V_m+\const_{a,b,m}.
\end{equation}
Here $y_1,\dots,y_m$ is an $m$-tuple of variables, $V_m=V_m(y_1,\dots,y_m)$ is
the Vandermonde, $\De^\baM_{\text{\rm variable\,} y_i}$ denotes the one-variate
Hahn operator acting on the $i$th variable, and the constant is given by
\eqref{eq9.F}.  The same argument as above shows that the operator $\De^\baM_m$
is well defined on the space of symmetric polynomials and kills the constants.

Alternatively, $\De^\baM_m$ can be interpreted as an operator acting on the
space of functions on $m$-point configurations
$\K=(k_1>\dots>k_m)\subseteq(N^m)$ (here we write $(k_1,\dots,k_m)$ instead of
$(y_1,\dots,y_m)$). This is just the interpretation that we need.

On the other hand, the difference operator $D_{m,m+1,0,b\mid N}$ acts on the
functions defined on set of the diagrams $\la$ or, equivalently, on the set of
configurations $\L$.

Now we use the correspondence $\L\leftrightarrow \K$ to compare the both
operators.

\begin{proposition}\label{prop9.A}
Under the correspondence $\la\leftrightarrow\L\leftrightarrow
\K\leftrightarrow\ka$, the operator $D_{m,m+a,0,b\mid N}$ turns into the
operator $\De^\baM_m$.
\end{proposition}

\begin{proof}
Let us regard $D_{m,m+a,0,b\mid N}$ as an operator on the space of functions
$F(\K)$. Then Lemma \ref{lemma9.D} shows that $D_{m,m+a,0,b\mid N}$ acts as the
following difference operator
$$
(D_{m,m+a,0,b\mid N} F)(\K)=\sum_{y\in\K}\sum_{\epsi=\pm1}\wt q(y\to
y+\epsi)[F(\K-\{y\}+\{y+\epsi\})-F(\K)].
$$
Looking at the explicit expressions for the jump rates $\wt q(y\to y+\epsi)$
given in Proposition \ref{prop9.A} and comparing them with the definition of
$\De^\baM_m$ (see \eqref{eq9.J}) we conclude that $D_{m,m+a,0,b\mid
N}=\De^\baM_m$.
\end{proof}

\subsection{Step 3: The transformation Hahn $\to$ Jacobi}\label{sect9.A}

Let us collect a few classic formulas about the Hahn and Jacobi orthogonal
polynomials. They can be found, e.g., in Koekoek-Swarttouw \cite{KS}.

The \emph{Hahn polynomials} with parameters $(a,b,M)$, denoted here by
$H^\baM_n(y)$, are the orthogonal polynomials on $\I=\{0,\dots,M\}$ with the
weight
$$
W^\baM_{\text{\rm Hahn}}(y)=\binom{b+y}y \binom{a+M-y}{M-y}, \qquad y\in\I.
$$
The subscript $n$ is the degree; it ranges also over $\I$. As was already
pointed out, our notation slightly differs from that of \cite{KS}: our
parameters $(a,b)$ correspond to parameters $(\be,\al)$ in \cite[Section
1.5]{KS}.

The Hahn polynomials form an eigenbasis for the Hahn difference operator
$\De^\baM$ defined in \eqref{eq9.J}:
\begin{equation}\label{eq9.N}
\De^\baM H^\baM_n=-n(n+b+a+1)H^\baM_n.
\end{equation}

Here is the explicit expression of the Hahn polynomials through a terminating
hypergeometric series of type $(3,2)$ at point $1$:
$$
H^\baM_n(y)={}_3 F_2\left[\begin{matrix}-n,\, n+b+a+1,\, -y\\
b+1,\, -M\end{matrix}\,\Biggl|\,1\right], \qquad n=0,\dots,M.
$$

Our notation for the Jacobi polynomials is $J^\ba_n(t)$; these are the
orthogonal polynomials on the unit interval $[0,1]$ with the weight
$$
W^\ba_{\text{\rm Jacobi}}(t)=t^b(1-t)^a, \qquad 0\le t\le 1.
$$
Note that many sources, including \cite{KS}, take the weight function
$(1-x)^a(1+x)^b$ with $x$ ranging over $[-1,1]$. The passage from $[0,1]$ to
$[-1,1]$ is given by the change of variable $x=2t-1$.

The Jacobi polynomials form an eigenbasis for the Jacobi difference operator:
\begin{equation}\label{eq9.O}
D^\ba J^\ba_n=-n(n+b+a+1)J^\ba_n, \qquad n=0,1,2,\dots\,.
\end{equation}

The Jacobi polynomials are expressed through the Gauss hypergeometric series:
$$
J^\ba_n(t)={}_2 F_1\left[\begin{matrix}-n,\, n+b+a+1\\
b+1\end{matrix}\,\Biggl|\,t\right], \qquad  n=0,1,2,\dots\,.
$$
Note that our normalization of the Jacobi polynomials differs from the
conventional one, but this is convenient for the computation below.

\begin{lemma}\label{lemma9.B}
The following relation holds
$$
\sum_{k=0}^M\binom{M}{k}t^k(1-t)^{M-k}H^\baM_n(k)=J^\ba_n(t), \qquad
n=0,\dots,M.
$$
\end{lemma}

\begin{proof}
This is checked directly using the explicit expressions for the polynomials.
Indeed, the sum in the left-hand side equals
$$
\sum_{k=0}^M\sum_{p=0}^n\frac{M!t^k(1-t)^{M-k}(-n)_p(n+b+a+1)_p(-k)_p}{k!(M-k)!(b+1)_p(-M)_pp!}.
$$

Let us change the order of summation and observe that $(-k)_p$ vanishes unless
$k\ge p$. Then the above expression can be rewritten as
$$
\sum_{p=0}^n\sum_{k=p}^M\frac{M!t^k(1-t)^{M-k}(-n)_p(n+b+a+1)_p(-k)_p}{k!(M-k)!(b+1)_p(-M)_pp!}.
$$

Next, let us set $q=k-p$ and observe that
$$
\frac{M!(-k)_p}{k!(M-k)!(-M)_p}=\frac{M!k!(M-p)!}{k!(k-p)!M!(M-k)!}=\binom{M-p}q.
$$

It follows that our double sum equals
$$
\sum_{p=0}^n\frac{(-n)_p(n+b+a+1)_p
}{(b+1)_pp!}\,t^p\,\sum_{q=0}^{M-p}\binom{M-p}qt^q(1-t)^{M-p-q}.
$$

The interior sum equals 1, so that we finally get
$$
\sum_{p=0}^n\frac{(-n)_p(n+b+a+1)_p }{(b+1)_pp!}\,t^p=J^\ba_n(t),
$$
as desired.
\end{proof}

The \emph{$m$-variate Hahn polynomials} are given by
$$
H^\baM_\nu(y_1,\dots,y_m)=\frac{\det\left[H^\baM_{n_j}(y_i)\right]}{V_m(y_1,\dots,y_m)}.
$$
Here $\nu$ is an arbitrary Young diagram contained in $(N^m)$ and
$$
n_j:=\nu_j+m-j, \qquad j=1,\dots,m.
$$
The definition is correct because the largest index $n_1$ does not exceed $M$
(recall that $M=N+m-1$; therefore, $\nu\subseteq(N^m)$ implies
$n_1=\nu_1+m-1\le M$).

Likewise, the \emph{$m$-variate Jacobi polynomials} are given by
$$
J^\ba_\nu(t_1,\dots,t_m)=\frac{\det\left[J^\ba_{n_j}(t_i)\right]}{V_m(t_1,\dots,t_m)}.
$$
Here $\nu$ is an arbitrary Young diagram with at most $m$ nonzero rows.

\begin{proposition}\label{prop9.C}
For every $N=1,2,\dots$ and every Young diagram $\nu\subseteq(N^m)$, the
operator $\La^\infty_N$ takes the Hahn polynomial $H^\baM_\nu$ to the
respective Jacobi polynomial $J^\ba_\nu$, within a constant factor.
\end{proposition}

\begin{proof}
By virtue of Proposition \ref{prop9.B},
\begin{multline}\label{eq9.K}
\La^\infty_N H^\baM_\nu)(t_1,\dots,t_m)\\
=\frac{\const_{m,M}}{V(t_1,\dots,t_m)}\sum_{M\ge
k_1>\dots>k_m\ge0}\det\left[\binom{M}{k_j}t_i^{k_j}(1-t_i)^{M-k_j}\right]
\det\left[H^\baM_{n_i}(k_j)\right].
\end{multline}

Now we apply a well-known identity, which is a consequence of the Cauchy-Binet
identity:
$$
\sum_{M\ge k_1>\dots>k_m\ge0}\det[f_i(k_j)]_{i,j=1}^m\det[g_i(k_j)]_{i,j=1}^m=
\det[h_{ij}]_{i,j=1}^m,
$$
where
$$
h_{ij}:=\sum_{k=0}^M f_i(k)g_j(k).
$$
It tells us that the sum in \eqref{eq9.K} equals the determinant of the
$m\times m$ matrix whose $(i,j)$ entry is
$$
\sum_{k=0}^M\binom{M}{k}t_i^{k}(1-t_i)^{M-k}H_{n_j}(k).
$$
By Lemma \ref{lemma9.B}, the last sum equals $J^\ba_{n_j}(t_i)$. This completes
the proof of the proposition.
\end{proof}

\subsection{Completion of proof}

As pointed out above (see \eqref{eq9.N} and \eqref{eq9.O}), the classic Hahn
and Jacobi polynomials are eigenfunctions of the respective operators, and the
$n$th eigenvalue in both cases is the same number $c(n):=-n(n+a+b+1)$.

By the very definition of the multivariate polynomials and operators, the
similar assertion holds  for arbitrary $m$ as well, and the eigenvalue
corresponding to a given label $\nu$ is equal to
$$
\sum_{i=1}^m[c(\nu_i+m-i)-c(m-i)].
$$

Combining this with the result of Step 3 (Proposition \ref{prop9.C}) we obtain
the desired commutation relation \eqref{eq9.D} which says that the link
$\La^\infty_N$ intertwines the Jacobi differential operator $D^{(a,b)}_m$ with
the Hahn difference operator $\De^{(a,b,N+m-1)}_m$.

Finally, as pointed out in the end of Subsection \ref{sect9.B}, the result of
Step 2 (Proposition \ref{prop9.A}) reduces Theorem \ref{thm9.A} to that
commutation relation.

This completes the proof of Theorem \ref{thm9.A}, which in turn implies Claim
\ref{claim2}. Thus, the proof of Theorem \ref{thm7.A} is completed.

\section{Appendix: uniform boundedness of multiplicities}\label{sect10}

Here we prove the statement used in the proof of Proposition \ref{prop2.F},
step 1. We formulate the result in a greater generality, which seems to be more
natural.

Let $\wt G$ be a connected reductive complex group and $G\subset\wt G$ be a
reductive subgroup. We assume $G$ is spherical, meaning that for any simple
$\wt G$-module $V$, the space $V^G$ of $G$-invariants has dimension at most 1.
For a simple $G$-module $W$ we write
$$
[V:W]:=\dim \operatorname{Hom}_G(W,V).
$$

\begin{proposition}\label{prop10.A}
Let $\wt G$, $G$, $V$, and $W$ be as above. If $W$ is fixed, then for the
multiplicity $[V:W]$ there exists a bound $[V:W]\le\const$, where the constant
depends only on $W$ but not on $V$.
\end{proposition}

The fact that we needed in Proposition \ref{prop2.E} is a particular case of
Proposition \ref{prop10.A} corresponding to $\wt G=GL(M+N,\C)$ and
$G=GL(M,\C)\times GL(N,\C)$.

\begin{proof}[First proof {\rm(}communicated  by Vladimir L. Popov{\rm)}]
Let us fix a Borel subgroup $B\subset\wt G$ and denote by $N$ the unipotent
radical of $B$. Let $A=\C[\wt G/N]$ be the algebra of regular functions on $\wt
G/N$. In other words, $A$ consists of holomorphic functions on $\wt G/N$ which
are $\wt G$-finite with respect to the action of $\wt G$ by left shifts. As a
$\wt G$-module, $A$ is the multiplicity free direct sum of all simple $\wt
G$-modules:
\begin{equation}\label{eq10.A}
A=\bigoplus_{\la\in\La_+}A_\la,
\end{equation}
where $\La_+$ denotes the additive semigroup of dominant weights with respect
to $B$ and $A_\la$ denotes the subspace of $A$ carrying the simple $\wt
G$-module with highest weight $\la$.

We fix a simple $G$-module $W$. Given a $G$-module $X$, we denote by $X^{(W)}$
the $W$-isotypic component in $X$. Using this notation, the desired claim can
be reformulated as follows: as $\la$ ranges over $\La_+$, the quantities $\dim
A_\la^{(W)}$ are uniformly bounded from above.

\emph{Step} 1. Let $A^G\subset A$ be the subalgebra of $G$-invariants.
Obviously, $A^{(W)}$ is a $A^G$-module. We claim that it is finitely generated.

Indeed, this is equivalent to saying that $\Hom_G(W,A)$ is finitely generated
as a $A^G$-module.

Observe that the expansion \eqref{eq10.A} is a grading of $A$. That is,
\begin{equation}
A_{\la'}A_{\la''}\subseteq A_{\la'+\la''}, \qquad \la',\la''\in\La_+.
\end{equation}
Since the semigroup $\La_+$ is finitely generated, the algebra $A$ is finitely
generated.

This property together with the fact that $G$ is assumed to be reductive make
it possible to apply the classic trick (used in Hilbert's theorem on
invariants) to the $A$-$G$-module $\Hom(W,A)$, see Popov-Vinberg \cite[Theorems
3.6 and 3.25]{PV}. Then we obtain that $(\Hom(W,A))^G$ is a finitely generated
$A^G$-module, as desired.

\emph{Step} 2. By virtue of Step 1, there exists a finite collection of weights
$\la(1),\dots, \la(n)\in\La_+$ such  that $A^{(W)}$ is generated over $A^G$ by
the subspace $ A_{\la(1)}^{(W)}+\dots+A_{\la(n)}^{(W)}$. {}From this and
\eqref{eq10.A} we conclude that for every weight $\la\in\La_+$, the subspace $
A_\la^{(W))}$ is contained in the sum of subspaces of the form
$A_{\la-\la(i)}^G A_{\la(i)}^{(W)}$, where $i\in\{1,\dots,n\}$ should be such
that $\la-\la(i)\in\La_+$.

Because $G$ is a spherical subgroup of $\wt G$, every subspace
$A_{\la-\la(i)}^G$ has dimension at most 1. This gives us the desired bound
$$
\dim A_\la^{(W)}\le\sum_{i=1}^n\dim A_{\la(i)}^{(W)},
$$
uniform on $\la\in\La_+$.
\end{proof}

\begin{proof}[Second proof {\rm(}sketch{\rm)}]
Given a finite-dimensional $G$-module $Y$, we can define the induced $\wt
G$-module $\Ind(Y)$: its elements are holomorphic vector-functions $f:\wt G\to
Y$, which are $\wt G$-finite with respect to right shifts and such that $f(g\wt
g)=gf(\wt g)$ for any $g\in G$ and $\wt g\in\wt G$.

As above, we fix a simple $G$-module $W$. The desired claim is equivalent to
the existence of a uniform bound for $[\Ind(W):V]$, the multiplicity of an
arbitrary simple $\wt G$-module $V$ in the decomposition of $\Ind(W)$.

Given a finite-dimensional $\wt G$-module $X$, let us denote by $X_G$ the same
space regarded as a $G$-module. One can choose $X$ in such a way that $X_G$
contained $W$. Then we obviously have $[\Ind(W):V]\le [\Ind(X_G):V]$.

The key observation is that $\Ind(X_G)$ is isomorphic to $\Ind(\C)\otimes X$,
where $\C$ stands for the trivial one-dimensional $G$-module.

Now let $V$ be an arbitrary simple $\wt G$-module. We have
$$
[\Ind(\C)\otimes X:V]=\dim\Hom_{\wt G}(V\otimes X^*, \Ind(\C)),
$$
where $X^*$ is the dual module to $X$. Observe that in the decomposition of
$V\otimes X^*$ on simple components, every multiplicity does not exceed $\dim
X^*=\dim X$ (this follows from a well-known formula describing the
decomposition of tensor products, see Zhelobenko \cite[end of \S124]{Zhe} or
Humphreys \cite[\S24, Exercise 9]{Humphreys} or else can be easily proved
directly). Since $\Ind(\C)$ is multiplicity free, we finally conclude that
$[\Ind(W):V]\le\dim X$, which is the desired uniform bound.
\end{proof}

\end{document}